\documentclass[11pt]{amsart}

\usepackage{geometry}
\usepackage{graphicx}
\usepackage{amssymb}
\usepackage{caption}
\usepackage{subcaption}
\usepackage[toc,page]{appendix}
\usepackage{enumitem}

\usepackage[all]{xy}

\newtheorem{theorem}{Theorem}

\newtheorem*{remark*}{Remark}

\newtheorem{lemma}{Lemma}
\newtheorem{example}{Example}

\renewcommand{\figurename}{Figure }

\usepackage[numbers,sort&compress]{natbib}
\bibpunct{[}{]}{;}{n}{,}{,}

\usepackage[colorlinks=false,bookmarksdepth=0]{hyperref}

\newcommand{\bfi}{\bfseries\itshape}

\allowdisplaybreaks

\title[Lagrangian and Hamiltonian Taylor Variational Integrators]{Lagrangian and Hamiltonian\\ Taylor Variational Integrators}
\author{Jeremy Schmitt, Tatiana Shingel, and Melvin Leok}

\begin{document}
\begin{abstract}

In this paper, we present a variational integrator that is based on an approximation of the Euler--Lagrange boundary-value problem via Taylor's method. This can viewed as a special case of the shooting-based variational integrator introduced in \cite{LeSh2011_sbvi}. The Taylor variational integrator exploits the structure of the Taylor method, which results in a shooting method that is one order higher compared to other shooting methods based on a one-step method of the same order. In addition, this method can generate quadrature nodal evaluations at the cost of a polynomial evaluation, which may increase its efficiency relative to other shooting-based variational integrators. A symmetric version of the method is proposed, and numerical experiments are conducted to exhibit the efficacy and efficiency of the method.
\end{abstract}
\maketitle

\section{Introduction}
This paper is concerned with the systematic construction and analysis of Lagrangian and Hamiltonian variational integrators of arbitrarily high-order derived from an underlying Taylor integrator. This can be viewed, on the Lagrangian side, as a special case of the shooting-based variational integrators introduced in \cite{LeSh2011_sbvi}, which provided a general framework for constructing a Lagrangian variational integrator from a given one-step method.

The main limitation of the shooting-based variational integrator approach is that in order to achieve higher-order accuracy, one requires multiple steps of the underlying one-step method in order to obtain approximations of the solution of the Euler--Lagrange boundary-value problem at the quadrature points. This is of course the best one can hope to achieve given a generic one-step method, but for one-step methods such as collocation methods or Taylor methods, one obtains a continuous approximation that can be evaluated at multiple points. As such, these methods only require a single step of the one-step method in order to obtain a continuous approximation of the Euler--Lagrange boundary-value problem that can be used to construct discrete Lagrangians and discrete Hamiltonians that generate symplectic integrators.

We focus on the use of Taylor integrators as the underlying one-step method, since they can be efficiently implemented to arbitrarily high-order for a broad range of problems by leveraging automatic differentiation techniques, and the resulting solution can be evaluated at additional quadrature points at the cost of a polynomial evaluation.

\section{Discrete Mechanics}\index{discrete mechanics}
Discrete Lagrangian mechanics \cite{MaWe2001} is based on a discrete analogue of Hamilton's principle, referred to as  the {\bfi discrete Hamilton's principle},
\[ \delta \mathbb{S}_d = 0,\]
where the {\bfi discrete action sum}, $\mathbb{S}_d:Q^{n+1}\rightarrow \mathbb{R}$, is given by
\[ \mathbb{S}_d(q_0,q_1,\ldots,q_n) = \sum\nolimits_{i=0}^{n-1} L_d(q_i, q_{i+1}) .\]
The {\bfi discrete Lagrangian}, $L_d:Q\times Q\rightarrow \mathbb{R}$, is a generating function of the symplectic flow, and is an approximation to the {\bfi exact discrete Lagrangian},
\begin{equation}
L_d^E(q_0,q_1;h)=\int_0^h L(q_{01}(t),\dot q_{01}(t)) dt,\label{exact_Ld_Jacobi}
\end{equation}
where $q_{01}(0)=q_0,$ $q_{01}(h)=q_1,$ and $q_{01}$ satisfies the
Euler--Lagrange equation in the time interval $(0,h)$.

The discrete variational principle yields the {\bfi discrete Euler--Lagrange (DEL)} equation,
\begin{equation}\label{DEL}
 D_2 L_d(q_{k-1},q_k)+D_1 L_d(q_k,q_{k+1})=0,
\end{equation}
which implicitly defines the {\bfi discrete Lagrangian map} $F_{L_d}:(q_{k-1},q_k)\mapsto(q_k,q_{k+1})$ for initial conditions $(q_0,q_1)$ that are sufficiently close to the diagonal of $Q\times Q$. This is equivalent to the {\bfi implicit discrete Euler--Lagrange (IDEL)} equations,
\begin{equation}
p_k=-D_1 L_d(q_k, q_{k+1}),\qquad p_{k+1}=D_2 L_d(q_k, q_{k+1}),\label{IDEL}
\end{equation}
which implicitly defines the {\bfi discrete Hamiltonian map} $\tilde{F}_{L_d}:(q_k,p_k)\mapsto(q_{k+1},p_{k+1})$, where the discrete Lagrangian is the Type I generating function of the symplectic transformation. Furthermore, the discrete Hamiltonian map associated with the exact discrete Lagrangian $\tilde{F}_{L_d^E}$ is the time-$h$ flow map of the Hamiltonian vector field. These observations serve as the basis by which the variational error analysis result of \S\ref{vea} is proven in \cite{MaWe2001}. In particular, variational error analysis relates the order to which a computable discrete Lagrangian approximates the exact discrete Lagrangian with the order of accuracy of the discrete Hamiltonian map when viewed as a one-step method for approximating the flow of Hamilton's equations.

\subsection{Variational error analysis\label{vea}}
The natural setting for analyzing the order of accuracy of a variational integrator is the variational error analysis framework introduced in \cite{MaWe2001}. In particular, Theorem 2.3.1 of \cite{MaWe2001} states that if a discrete Lagrangian, $L_d:Q\times Q\rightarrow\mathbb{R}$, approximates the exact discrete Lagrangian, $L_d^E:Q\times Q\rightarrow\mathbb{R}$, given in \eqref{exact_Ld_Jacobi} to order $p$, i.e.,
\begin{align} L_d(q_0, q_1;h)=L_d^E(q_0,q_1;h)+\mathcal{O}(h^{p+1}),\label{order_Ld}
\end{align}
then the discrete Hamiltonian map, $\tilde{F}_{L_d}:(q_k,p_k)\mapsto(q_{k+1},p_{k+1})$, viewed as a one-step method, is order $p$ accurate.

\section{Lagrangian Taylor Variational Integrator}
The exact discrete Lagrangian~\eqref{exact_Ld_Jacobi} is given by the action integral evaluated along the solution of the Euler--Lagrange boundary-value problem. In turn, the boundary-value problem with boundary data $(q_0,q_1)$ can be related to an initial-value problem with initial data $(q_0,v_0)$, which satisfies the condition $q_1=\pi_Q \Phi_h(q_0,v_0)$, where $\pi_Q:TQ\rightarrow Q$ is the canonical projection onto $Q$ and $\Phi_h:TQ\rightarrow TQ$ is the exact time-$h$ flow map. This yields the following characterization of the exact discrete Lagrangian,
\[L_d^E(q_0,q_1;h)=\int_0^h L(\Phi_t(q_0,v_0)) dt,\]
where $q_1=\pi_Q \Phi_h(q_0,v_0)$. The Taylor variational integrator is generated by a computable discrete Lagrangian obtained when the integral is approximated by a quadrature rule, and the Taylor method is used to approximate the flow map that relates the boundary data $(q_0,q_1)$ with the initial-value data $(q_0,v_0)$, and the trajectory associated with the initial data. The following summarizes the construction of the Taylor variational integrator.
\begin{enumerate}[label=(\roman*)]
\item The approximation to $\dot{q}(0)=v_0$, denoted as $\tilde{v}_0$, is defined via the inverse problem, \begin{align} q_1=\pi_Q \circ \Psi^{(r+1)}_h(q_0,\tilde{v}_0), \label{tildev0a} \end{align} where $\pi_Q:TQ\rightarrow Q$ is the canonical projection onto $Q$ and $\Psi^{(r+1)}_h:TQ\rightarrow TQ$ denotes a $(r+1)$-order Taylor method.
\item Generate approximations to the quadrature nodal values, $q_{c_i}\approx q(c_ih)$ (excluding $q_1$ if needed, which is assumed to be given) and $v_{c_i}\approx \dot{q}(c_ih)$, via Taylor's method using $\tilde{v}_0$, \begin{align} (q_{c_i},v_{c_i}) &= \Psi^{(r)}_{c_ih}(q_0,\tilde{v}_0). \label{qcvc} \end{align}
\item Apply the quadrature rule to construct the associated discrete Lagrangian, \begin{align}L_d(q_0,q_1;h) = h \sum_{i=1}^m b_i L(q_{c_i},v_{c_i}). \label{Ld_Taylor} \end{align}
\item Applying the discrete Legendre transforms implicitly defines the method,
	\begin{align*}
		p_0 &= -D_1L_d(q_0,q_1;h), \\ 
		p_1 &= D_2L_d(q_0,q_1;h).
	\end{align*}
\end{enumerate}

\begin{remark*}
It may seem like a waste to solve for $\tilde{v}_0$ using a $(r+1)$-order Taylor method, and then to use only a $r$-order method to solve the Euler--Lagrange boundary-value problem, but from an implementation perspective, no additional derivative evaluations are needed to solve \eqref{tildev0a}, other than those already required in implementing the $r$-order Taylor method on $TQ$. In fact, it is an efficient use of the higher-derivative information we already needed to compute in order to construct the $r$-order Taylor method on $TQ$.

This apparent discrepancy can be resolved by thinking of equation \eqref{tildev0a} as being a $(r+1)$-order Taylor method for the second-order differential equation on $Q$, and \eqref{qcvc} as a $r$-order Taylor method on the first-order differential equation on $TQ$. In particular, notice that because of the canonical projection $\pi_Q$ in equation \eqref{tildev0a}, we only need to compute up to $q^{(r+1)}(0)$ in order to solve for $\tilde{v}_0$, instead of the up to  $q^{(r+2)}(0)$ that is necessary to define $\Psi^{(r+1)}_h$. But, we needed to compute up to $v^{(r)}(0)=q^{(r+1)}(0)$ in order to construct $\Psi^{(r)}_h$, the $r$-order Taylor method on $TQ$.
\end{remark*}

The following lemmas are needed for a theorem on the accuracy of the method. These lemmas can be proved using Lipschitz continuity and triangle inequalities (see Appendices for their proofs).
\begin{lemma}\label{lemma:velocity_approximation}
$\tilde{v}_0$ as defined by, \eqref{tildev0a}, approximates $v_0$ to at least $\mathcal{O}(h^{r+1})$.
\end{lemma}

\begin{lemma}\label{lemma:bvp_approximation}
A $r$-order Taylor method with initial conditions $(q_0,\tilde{v}_0)$, where $\tilde{v}_0$ is defined by \eqref{tildev0a}, is accurate to at least $\mathcal{O}(h^{r+1})$ for the Euler--Lagrange boundary-value problem with boundary conditions $(q_0,q_1)$.
\end{lemma}

\begin{theorem}
Assuming a Lagrangian $L$ that is Lipschitz continuous in both variables, then for a $r$-order accurate Taylor method, $\Psi^{(r)}_h$, and a $s$-order accurate quadrature formula, the associated Taylor discrete Lagrangian \eqref{Ld_Taylor} has order of accuracy at least $\min(r+1,s)$.
\end{theorem}
\begin{proof}
$(q_d(t), v_d(t))$, associated with the Taylor method $\Psi_h$ of order $r$ and initial data $(q_0,\tilde{v}_0)$, approximates the exact solution  $(q_{01}(t),v_{01}(t))$ of the Euler--Lagrange boundary-value problem with the following error,
\begin{align*}
q_{01}(c_i h)&=q_d(c_i h)+\mathcal{O}(h^{r+1}),\\
v_{01}(c_i h)&=v_d(c_i h)+\mathcal{O}(h^{r+1}).
\end{align*}
If the numerical quadrature formula is order $s$ accurate, then 
\begin{align*}
L_d^E(q_0, q_1;h)&=\int_0^h L(q_{01}(t),v_{01}(t))dt\\
&=\left[h\sum\nolimits_{i=1}^m b_i L(q_{01}(c_i h), v_{01}(c_i h))\right]+\mathcal{O}(h^{s+1})\\
&=\left[h\sum\nolimits_{i=1}^m b_i L(q_d(c_i h)+\mathcal{O}(h^{r+1}), v_d(c_i h)+\mathcal{O}(h^{r+1}))\right]+\mathcal{O}(h^{s+1})\\
&=\left[h\sum\nolimits_{i=1}^m b_i L(q_d(c_i h), v_d(c_i h))\right]+\mathcal{O}(h^{r+2})+\mathcal{O}(h^{s+1})\\
&=L_d(q_0, q_1;h)+\mathcal{O}(h^{r+2})+\mathcal{O}(h^{s+1})\\
&=L_d(q_0, q_1;h)+\mathcal{O}(h^{\min(r+1,s)+1}),
\end{align*}
where we used the quadrature approximation error, the error estimates on the shooting solution, and the assumption that $L$ is Lipschitz continuous in both variables.
\end{proof}
The choice of the Taylor method as the underlying one-step method has the advantage that it only requires one to precompute the prolongation of the Euler--Lagrange vector field once at the initial time, and the computational cost is not increased appreciably by having to compute the numerical solution at multiple quadrature nodes, since that only requires a polynomial evaluation. This efficiency in evaluation improves upon the methods outlined in \cite{LeSh2011} and \cite{LeSh2011_sbvi}, which utilized collocation and the shooting-method, respectively.

\begin{example}
Consider a first-order Taylor variational integrator that uses the rectangular quadrature rule about the initial point. We assume a Lagrangian of the form $L(q,\dot{q})=\frac{1}{2}\dot{q}^TM\dot{q}-V(q)$. Then the integrator is constructed as follows:
\begin{enumerate}[label=(\roman*)]
\item The inverse problem is, $$q_1 = q_0 + h\tilde{v}_0.$$ This implies $\tilde{v}_0 = \frac{q_1-q_0}{h}$, where $q_0$, $q_1$ are the given boundary conditions.
\item The quadrature nodal values are $q_{c_1} = q_0$ and $v_{c_1} = \tilde{v}_0 = \frac{q_1-q_0}{h}$. 
\item The corresponding discrete Lagrangian is given by,
	\begin{align*}
		L_d(q_0,q_1;h) &= hL\Big(q_0,\frac{q_1-q_0}{h}\Big) \\ 
		&= h\left[\frac{1}{2}\Big(\frac{q_1-q_0}{h}\Big)^TM\Big(\frac{q_1-q_0}{h}\Big) - V(q_0)\right] .
	\end{align*}
\item The discrete Legendre transforms are given by,
	\begin{align*}
		p_0 &= M\Big(\frac{q_1-q_0}{h}\Big) +h \nabla V(q_0), \\ 
		p_1 &= M\Big(\frac{q_1-q_0}{h}\Big).
	\end{align*}
\end{enumerate}
With some rearranging and substitution we see that this is symplectic Euler-A,
\begin{align*}
	q_1 &= q_0 +hM^{-1}p_1, \\
	p_1 &= p_0 -h\nabla V(q_0).
\end{align*}
\end{example}

If we use the rectangular quadrature rule about the end point, then the resulting method would be symplectic Euler-B. If instead we choose the trapezoid quadrature rule, then the resulting method will be St\"ormer--Verlet. All three of these classic symplectic integrators can be derived as Taylor variational integrators. However, there are also novel methods that come from the Taylor variational integrator framework, as the next example illustrates.

\begin{example}
Consider a second-order Taylor variational integrator, which utilizes a first-order Taylor method combined with the trapezoid rule to approximate the discrete Lagrangian. The approximate initial velocity is given by,
\begin{align*}
\tilde{v}_0 = \frac{q_1-q_0}{h} -\frac{h}{2}M^{-1}\nabla V(q_0).
\end{align*}
The resulting method is an explicit second-order method given by,
\begin{align*}
q_1 &= q_0 + h M^{-1} p_0 - \frac{h^2}{2} M^{-1} \nabla V(q_0) + \frac{h^4}{4} M^{-1} \nabla \nabla V(q_0) M^{-1} V(q_0), \\
p_1 &= M \tilde{v}_0 -\frac{h}{2}(\nabla V(q_0) + \nabla V(q_1)).
\end{align*}
\end{example}

As demonstrated above, Lagrangian Taylor variational integrators provide a very general family of symplectic integrators that include not only classic symplectic integrators, but also novel symplectic integrators. The Taylor variational integrator is amenable to the construction of higher-order symplectic integrators that can benefit from many of the numerical techniques that have enhanced the classical Taylor method (see \cite{JoZo2005}, \cite{Simo2007}). In particular, automatic differentiation allows for accurate and relatively cheap derivative evaluations (see \cite{Haro2011}, \cite{Ral1981}, \cite{Neid2010}). In general, higher-order Taylor variational integrators will require solving a system of nonlinear equations, which can be dealt with using standard methods (see \cite{Iserles2009}). While it is clear that Taylor variational integrators will have a higher computational cost than the Taylor method, in many cases the Taylor variational integrator can preserve accuracy and structure for larger step sizes, which may justify the higher cost per step. We will further examine these topics in section \ref{numerical_experiments}. Next, we consider discrete Hamiltonian formulations and symmetric formulations of the Taylor variational integrator.

\section{Hamiltonian and Symmetric Taylor Variational Integrators}
\subsection{Hamiltonian Taylor Variational Integrators}\label{discHam}
Thus far, we have derived the Taylor variational integrator by approximating the discrete Lagrangian, which is a type I generating function of the symplectic map/integrator. However, we will also consider the discrete right and discrete left Hamiltonians (see \cite{LaWe2006}, \cite{LeZh2011}), which are type II and type III generating functions, respectively. The motivation being that for a degenerate Hamiltonian there may be no corresponding Lagrangian formulation, in which case the discrete Hamiltonian formulation may be the only way to construct a variational integrator. Also, it has recently been shown in \cite{ScLe2017} that even when the Legendre transform is a diffeomorphism, the discrete Lagrangian and discrete Hamiltonian formulation generated by a fixed approximation scheme can lead to different variational integrators.

 The boundary-value formulation of the exact discrete right Hamiltonian is given by,
\begin{align*}
H_d^{+,E}(q_0,p_1;h)& =  \left(p_1^T q_1 - \int_0^T \left[ p^T \dot{q}-H(q, p) \right]
dt\right),
\end{align*}
where $(q(t),p(t))$ satisfy Hamilton's equations with boundary conditions $q(0)=q_0$, $p(T)=p_1$.
Now let us consider the construction of a Taylor discrete right Hamiltonian.
\begin{enumerate}[label=(\roman*)]
\item Construct a $r$-order Taylor expansion on the cotangent bundle, $T^*Q$, and solve for $\tilde{p}_0$,
$$p_1 = \pi_{T^*Q} \circ \Psi^{(r)}_h(q_0,\tilde{p}_0),$$
where $\pi_{T^*Q}:(q,p)\mapsto p$.
\item Pick a quadrature rule of order $s$ with quadrature weights and nodes given by $(b_i,c_i)$ for $i=1,\ldots,m$.
\item Use a $r$-order Taylor method to generate approximations of $(q(t),p(t))$ at the quadrature nodes,
$$(q_{c_i},p_{c_i})= \Psi^{(r)}_{c_ih}(q_0,\tilde{p}_0),$$
and use a $(r+1)$-order Taylor method on the configuration manifold to generate the approximation to the boundary term $q_1$,
$$\tilde{q}_1=\pi_Q \circ \Psi^{(r+1)}_h(q_0,\tilde{p}_0).$$
\item Use the quadrature rule and approximate boundary term, $\tilde{q}_1$, to construct the discrete right Hamiltonian of order $\min(r+1,s)$,
$$H_d^+(q_0,p_1;h)=p_1^T\tilde{q}_1 - h \sum^m_{i=1} \Big[p_{c_i}^T\dot{q}_{c_i} - H\Big(\Psi^{(r)}_{c_ih}(q_0,\tilde{p}_0)\Big)\Big],$$
where $\dot{q}_{c_i}$ is obtained by inverting the continuous Legendre transform, $(q_{c_i},p_{c_i})=\mathbb{F}L(q_{c_i},\dot{q}_{c_i})$.
\item The method is implicitly defined by the implicit discrete right Hamilton's equations,
	\begin{equation}
		q_1 = D_2H_d^+(q_0, p_1), \qquad p_0 = D_1 H_d^+(q_0, p_1). \label{IRHE} 
	\end{equation}
\end{enumerate}

The boundary-value formulation of the exact discrete left Hamiltonian is given by,
\begin{align*}
H_d^{-,E}(q_1,p_0;h)& = - \left(p_0^T q_0 - \int_0^T \left[ p^T \dot{q}-H(q, p) \right]
dt\right),
\end{align*}
where $(q(t),p(t))$ satisfy Hamilton's equations with boundary conditions $q(T)=q_1$, $p(0)=p_0$.
Now let us consider the construction of a Taylor discrete left Hamiltonian.
\begin{enumerate}[label=(\roman*)]
\item Construct a $(r+1)$-order Taylor expansion on the cotangent bundle, $T^*Q$, and solve for $\tilde{q}_0$,
$$q_1 = \pi_{Q} \circ \Psi^{(r+1)}_h(\tilde{q}_0,p_0).$$
\item Pick a quadrature rule of order $s$ with quadrature weights and nodes given by $(b_i,c_i)$ for $i=1,\ldots,m$.
\item Use a $r$-order Taylor method to generate approximations of $(q(t),p(t))$ at the quadrature nodes,
$$(q_{c_i},p_{c_i})= \Psi^{(r)}_{c_ih}(\tilde{q}_0,p_0).$$
\item Use the quadrature rule and approximate boundary term, $\tilde{q}_0$, to construct the discrete left Hamiltonian of order $\min(r+1,s)$,
$$H_d^-(q_1,p_0;h)=-p_0^T\tilde{q}_0 - h \sum^m_{i=1} \Big[p_{c_i}^T\dot{q}_{c_i} - H\Big(\Psi^{(r)}_{c_ih}(\tilde{q}_0,p_0)\Big)\Big],$$
where $\dot{q}_{c_i}$ is obtained by inverting the continuous Legendre transform, $(q_{c_i},p_{c_i})=\mathbb{F}L(q_{c_i},\dot{q}_{c_i})$.
\item The method is implicitly defined by the implicit discrete left Hamilton's equations,
	\begin{equation}
		p_1 = -D_2H_d^-(q_1, p_0;h), \qquad q_0 = -D_1 H_d^-(q_1, p_0;h). \label{ILHE} 
	\end{equation}
\end{enumerate}

The St\"ormer--Verlet method can be derived as a Lagrangian Taylor variational integrator by choosing $r=0$ for the respective Taylor methods and using the trapezoid rule for the quadrature rule. This yields a discrete Lagrangian corresponding to the St\"ormer--Verlet method,
\begin{equation*}
L_d(q_0,q_1;h) = \frac{h}{2} \left( \Big(\frac{q_1-q_0}{h}\Big)^TM\Big(\frac{q_1-q_0}{h}\Big) - V(q_0) - V(q_1) \right).
\end{equation*}
Choosing $r=0$ and the trapezoid rule to construct a Hamiltonian Taylor variational integrator results in a discrete right Hamiltonian given by,
\begin{equation*}
H_d^+(q_0,p_1;h) = p_1^T (q_0 + hM^{-1}p_1) - \frac{h}{2} \left( p_1^TM^{-1}p_1 - V(q_0 + hM^{-1}p_1) \right),
\end{equation*}
and a discrete left Hamiltonian given by,
\begin{equation*}
H_d^-(q_1,p_0;h) = p_0^T (q_1 - hM^{-1}p_0) - \frac{h}{2} \left( p_0^TM^{-1}p_0 - V(q_1 - hM^{-1}p_0) \right).
\end{equation*}
The corresponding methods are not St\"ormer--Verlet, in fact they are neither symmetric nor explicit. However, a simple calculation shows that these discrete Hamiltonians are adjoint to each other (see \cite{ScLe2017} for info on adjoint discrete Hamiltonians), i.e. $-H_d^+(q_1,p_0;-h) = H_d^-(q_1,p_0;h)$. Therefore, a symmetric method can be constructed by composing the two methods. We will denote the resulting symmetric method by SVHd, and we compare it to St\"ormer--Verlet in section \ref{numerical_experiments} (see \figurename{\ref{SVHd}} and \figurename{\ref{figFPU}}).

It should be noted that some approximations schemes do yield the same method when applied to a discrete Lagrangian and a discrete right/left Hamiltonian. For instance, choosing $r=0$ and the rectangular rule about the end point will yield symplectic Euler-B for both the discrete Lagrangian and discrete right Hamiltonian approximation. When can we expect a fixed approximation scheme applied to a discrete Lagrangian and a discrete right Hamiltonian to yield the same method? The following theorem answers this question.

\begin{theorem} Assuming a regular Lagrangian, we consider a fixed approximation scheme used to construct a discrete Lagrangian, $L_d$, and a discrete right Hamiltonian, $H_d^+$. This results in two integrators, $\tilde{F}_{L_d}:(q_0,p_0) \mapsto (q_{1,L_d},p_{1,L_d})$ and $\tilde{F}_{H_d^+}:(q_0,p_0) \mapsto (q_{1,H_d^+},p_{1,H_d^+})$. If the discrete right Hamiltonian approximation satisfies $p_{1,H_d^+}=D_2L_d(q_0,\hat{q}_1)$, where $\hat{q}_1$ is the approximated value of $q_1$, then the integrators represent the same map, i.e., $(q_{1,L_d},p_{1,L_d}) = (q_{1,H_d^+},p_{1,H_d^+})$.
\end{theorem}

We have placed the proof of the above theorem in the appendix. It is important to note that even though the theorem guarantees the analytical equivalence of the integrators, this does not guarantee numerical equivalence (see \cite{ScLe2017}).

\subsection{Symmetric Lagrangian Taylor Variational Integrators}
Consider the following variational derivation of the St\"ormer--Verlet method. Construct the discrete Lagrangian by using the trapezoid rule and approximating $\tilde{v}_0$ and $\tilde{v}_1$ with the inverse problems given by, 
\begin{equation*}
q_1 = \Psi_h^{(1)}(q_0, \tilde{v}_0), \qquad q_0 = \Psi_{-h}^{(1)}(q_1, \tilde{v}_1),
\end{equation*}
where $\Psi^r_h$ denotes the $r$-th order Taylor method with step size $h$. Then, the velocity approximations are given by,
\begin{equation*}
\tilde{v}_0 = \frac{q_1-q_0}{h}, \qquad \tilde{v}_1 = \frac{q_1-q_0}{h},
\end{equation*}
and the resulting discrete Lagrangian yields the St\"ormer--Verlet method,
\begin{equation*}
L_d(q_0,q_1;h) = \frac{h}{2} \left( \Big(\frac{q_1-q_0}{h}\Big)^T M \Big(\frac{q_1-q_0}{h}\Big) - V(q_0) - V(q_1) \right).
\end{equation*}

It is well-known that St\"ormer--Verlet is a symmetric method, and that symmetric methods preserve important structure of time-reversible equations and are desirable for highly-oscillatory problems (see chapters V and XI of \cite{HaLuWa2006}). We can generalize the above approximation to yield a class of symmetric Taylor variational integrators. The approximation scheme uses a symmetric quadrature rule with weights and nodes $\{b_i,c_i\}_{i=1}^m$, and the Taylor method, and it is outlined as follows:
\begin{enumerate}[label=(\roman*)]
\item Solve the inverse problems for $\tilde{v}_0$ and $\tilde{v}_1$,
	\begin{equation} 
		q_1 = \Psi_h^{(r)}(q_0,\tilde{v}_0), \qquad q_0 = \Psi_{-h}^{(r)}(q_1\tilde{v}_1). \label{tildevs}
	\end{equation}
\item Generate approximations to the quadrature nodes $(q_{c_i},v_{c_i})$ via, \begin{align*} \label{q&v} q_{c_i} &= c_i \pi_Q \circ \Psi^{(r)}_{c_ih}(q_0,\tilde{v}_0) + (1-c_i) \pi_Q \circ \Psi^{(r)}_{-(1-c_i)h}(q_1,\tilde{v}_1) \\ v_{c_i} &= c_i \pi_Q \circ \Psi^{(r-1)}_{c_ih}(q_0,\tilde{v}_0) + (1-c_i) \pi_Q \circ \Psi^{(r-1)}_{-(1-c_i)h}(q_1,\tilde{v}_1).  \end{align*}
Note $q_0$, $q_1$, $\tilde{v}_0$, and $\tilde{v}_1$ are used as the approximations for their respective quadrature nodal values. Also, since the quadrature rule is assumed to be symmetric, $c_i=1-c_{m-i+1}$ and $b_i=b_{m-i+1}$.
\item Construct the discrete Lagrangian, \begin{equation*} L_d(q_0,q_1;h) = h \sum_{i=1}^m b_i L(q_{c_i},v_{c_i}) \end{equation*}
\item Apply the discrete Legendre transforms to implicitly define the variational integrator, \begin{equation*} p_0 = -D_1L_d(q_0,q_1;h), \qquad p_1 = D_2L_d(q_0,q_1;h). \end{equation*}
\end{enumerate}

\begin{theorem}
The symmetric Taylor variational integrator is a symmetric method.
\end{theorem}
\begin{proof}
By theorem 2.4.1 of \cite{MaWe2001}, it is sufficient and necessary to show that the discrete Lagrangian of the symmetric Taylor variational integrator is self-adjoint, i.e., $L_d(q_0,q_1;h)=-L_d(q_1,q_0;-h)$. We will use (*) to denote the approximated values resulting from exchanging $(q_0,q_1,h)$ for $(q_1,q_0,-h)$. Exchanging $(q_0,q_1,h)$ for $(q_1,q_0,-h)$ transforms \eqref{tildevs} into,
\begin{equation*}
q_0 = \Psi_{-h}^{(1)}(q_1,\tilde{v}^*_0), \qquad q_1 = \Psi_h^{(1)}(q_0,\tilde{v}^*_1),
\end{equation*}
so that $\tilde{v}^*_0=\tilde{v}_1$ and $\tilde{v}^*_1=\tilde{v}_0$. Therefore,
\begin{align*}
q^*_{c_i} &= c_i \pi_Q \circ \Psi^{(r)}_{-c_ih}(q_1,\tilde{v}^*_0) + (1-c_i) \pi_Q \circ \Psi^{(r)}_{(1-c_i)h}(q_0,\tilde{v}^*_1) \\
&= (1-c_{m-i+1}) \pi_Q \circ \Psi^{(r)}_{-(1-c_{m-i+1})h}(q_1,\tilde{v}_1) + c_{m-i+1} \pi_Q \circ \Psi^{(r)}_{c_{m-i+1}h}(q_0,\tilde{v}_0) \\
&= q_{c_{m-i+1}}.
\end{align*}
The second to last line follows from the fact that the quadrature rule is symmetric and therefore satisfies $1-c_i = c_{m-i+1}$. The same steps show that $v^*_{c_i}=v_{c_{m-i+1}}$. The symmetric quadrature rule also implies that $b_i=b_{m-i+1}$, so that we have the following,
\begin{align*}
-L_d(q_1,q_0;-h) &= -(-h) \sum_{i=1}^m b_i L(q^*_{c_i},v^*_{c_i}) \\
&= h \sum_{i=1}^m b_i L(q_{c_i},v_{c_i}) \\
&= L_d(q_0,q_1;h).
\end{align*}
\end{proof}

\begin{theorem}
Given a regular Lagrangian, an odd $r$-order Taylor method, and a symmetric quadrature rule of order $r+1$, then the resulting symmetric Taylor variational integrator is of order $r+1$.
\end{theorem}
\begin{proof}
Lemmas~\ref{lemma:velocity_approximation} and~\ref{lemma:bvp_approximation} imply that the nodal value approximations, $q_{c_i}$ and $v_{c_i}$, are of order $r$ and $r-1$ respectively. Therefore,
\begin{align*}
L_d(q_0,q_1;h) &= h \sum_{i=1}^m b_i L(q_{c_i},v_{c_i}) \\
&= h \sum_{i=1}^m b_i L(q(c_ih)+\mathcal{O}(h^{r+1}),\dot{q}(c_ih)+\mathcal{O}(h^r)) \\
&= h \sum_{i=1}^m b_i \left(L(q(c_ih),\dot{q}(c_ih)) + \mathcal{O}(h^r) \right) \\
&= h \sum_{i=1}^m b_i L(q(c_ih),\dot{q}(c_ih)) + \mathcal{O}(h^{r+1}) \\
&= \int^h_0 L(q(t),\dot{q}(t)) dt + \mathcal{O}(h^{r+1}) \\
&= L_d^E(q_0,q_1;h) + \mathcal{O}(h^{r+1}).
\end{align*}
We have used the order of the nodal approximations, the error order of the quadrature rule, and the Lipschitz continuity of a regular Lagrangian. By theorem 2.3.1 of \cite{MaWe2001}, the resulting variational integrator, denoted by $\tilde{\Psi}_h$, is at least of order $r$, i.e.,
\begin{align*}
\tilde{\Psi}_h(q_0,v_0) &= \Phi_h(q_0,v_0) + \mathcal{O}(h^{r+1}) \\
&= \Phi_h(q_0,v_0) + C(q_0,v_0) h^{r+1} + \mathcal{O}(h^{r+2}),
\end{align*}
where $\Phi_h$ is the true flow  of the Euler--Lagrange equations, and the last equality is a consequence of the implicit function theorem.

Finally, since the variational integrator is symmetric and $r+1$ is even, the method will be of order $r+1$ as the following implies.
\begin{align*}
\Phi_h(q_0,v_0) - C(q_0,v_0) h^{r+1} + \mathcal{O}(h^{r+2}) &= \Phi^*_h(q_0,v_0) - C(q_0,v_0) (-h)^{r+1} + \mathcal{O}(h^{r+2}) \\
&= \tilde{\Psi}^*_h(q_0,v_0) \\
&= \tilde{\Psi}_h(q_0,v_0) \\
&=\Phi_h(q_0,v_0) + C(q_0,v_0) h^{r+1} + \mathcal{O}(h^{r+2}),
\end{align*}
which implies $C(q_0,v_0) h^{r+1}=0$, and the method is of order $r+1$ as claimed.
\end{proof}

The symmetric Taylor variational integrator is of order $r+1$, but only requires the derivatives of a $r$-order Taylor method, which makes it more efficient than the non-symmetric  Taylor variational integrator, in addition to the qualitative benefits associated with its symmetry. However, applying this approximation scheme to generate a discrete Hamiltonian will not directly lead to a symmetric method. Recall that the symmetric Taylor variational integrator was inspired by St\"ormer--Verlet, so it is likely that using this approximation scheme to generate a discrete right and left Hamiltonian will result in the discrete left and right Hamiltonian methods that are adjoint to each other. In that case, the composition of these methods should yield a symmetric method from the discrete Hamiltonian formulation. We conjecture that if an approximation scheme yields a symmetric discrete Lagrangian, then the corresponding discrete right and left Hamiltonians will be adjoint. We will explore this further in future work.

\section{Numerical Implementation and Experiments}\label{numerical_experiments}
We now discuss the numerical implementation of the methods introduced in this paper. Below, we present the algorithm for the Lagrangian Taylor variational integrator, and we discuss some of our observations about the implementation details. Additionally, we compare the methods to other kinds of variational integrators, and discuss their relative merits.
\\

\noindent\underline{\textbf{Algorithm}}
Given $(q_0,p_0)$, $h$, $L(q(t),\dot{q}(t))$, the Euler--Lagrange vector field, quadrature weights and nodes $\{(b_i,c_i)\}_{i=1:m}$, and the desired order of the method $r+1$, then the Taylor variational integrator will output $(q_1,p_1)$ and is implemented as follows:
\begin{enumerate}
\item Prolongate the Euler--Lagrange vector field to obtain derivatives $q^{(j)}(q(t),v(t))$ for $j=1,\dots,r+1$.
\item Compute the partial derivatives $\frac{\partial q^{(j)}(q,v)}{\partial q}$ and $\frac{\partial q^{(j)}(q,v)}{\partial v}$.
\item Solve the following nonlinear system for $q_1$ and $\tilde{v}_0$: \begin{equation*} \begin{cases} 0 = q_1 - q_0 - h\tilde{v}_0 - \sum_{j=2}^{p+1} q^{(j)}(q_0,\tilde{v}_0) \frac{h^{j}}{j!}, \\ 0=p_0+\frac{\partial L_d(q_0,q_1)}{\partial q_0}. \end{cases} \end{equation*}
\item Finally, $p_1$ is given explicitly by,$$p_1=\frac{\partial L_d(q_0,q_1)}{\partial q_1}.$$\\
\end{enumerate}

When solving the nonlinear system that arises above, the following points should be noted:
\begin{enumerate}
\item In general, the nonlinear system is not amenable to a fixed-point iteration, so a form of Newton's method is preferable.
\item Each iteration will require evaluation of \begin{align*} q_{c_i} &= q_0 + h\tilde{v}_0 + \sum_{j=2}^r q^{(j)}(q_0,\tilde{v}_0)\frac{h^j}{j!}, \\ v_{c_i} &= \tilde{v}_0 + \sum_{j=2}^{r+1} q^{(j)}(q_0,\tilde{v}_0) \frac{h^{j-1}}{(j-1)!}. \end{align*}
\item The following requires computing $\frac{\partial \tilde{v}_0}{\partial q_0}$, \begin{align*} -p_0 &= \frac{\partial L_d(q_0,q_1)}{\partial q_0} \\ &= h \sum_{i=1}^m b_i\bigg(\frac{\partial L(q_{c_i},v_{c_i})}{\partial q_0} + \frac{\partial L(q_{c_i},v_{c_i})}{\partial \tilde{v}_0} \frac{\partial \tilde{v}_0}{\partial q_0}\bigg)^T \end{align*}
Fortunately, this can be found explicitly and need only be computed once at the beginning of the iteration, \begin{align*} \frac{\partial \tilde{v}_0}{\partial q_0} = \bigg(I + \sum^{r+1}_{j=2} \frac{\partial q^{(j)}(q_0,\tilde{v}_0)}{\partial \tilde{v}_0} \frac{(c_i h)^{j-1}}{j!}\bigg)^{-1} \bigg(\frac{-1}{h}I - \sum_{j=2}^{r+1} \frac{\partial q^{(j)}(q_0,\tilde{v}_0)}{\partial q_0} \frac{(c_i h)^{j-1}}{j!}\bigg) \end{align*}
\item Likewise, when solving $p_1=\frac{\partial L_d(q_0,q_1)}{\partial q_1}$, it will be necessary to compute \begin{align*} \frac{\partial \tilde{v}_0}{\partial q_1} = \frac{1}{h}\bigg(I + \sum_{j=2}^{r+1} \frac{\partial q^{(j)}(q_0,\tilde{v}_0)}{\partial \tilde{v}_0} \frac{(c_i h)^{j-1}}{j!}\bigg)^{-1}, \end{align*} which is explicit and is composed of terms that have already been computed. 
\end{enumerate}
Observe that good initial guesses for the nonlinear system are provided with little computational cost, by using a $(r+1)$-order Taylor method for $q_1$ and the Legendre transform of $p_0$ for $\tilde{v}_0$. Since this yields an approximate solution that is comparable in accuracy to the one obtained by the corresponding Taylor variational integrator, this yields a predictor-corrector implementation, where the Taylor variational integrator applies a symplectic correction that converges very rapidly. In general, when solving a nonlinear system as part of a symplectic method, the method becomes an almost symplectic method (see \cite{Tan2005}) unless it is solved to within machine precision. This implies that the error tolerance of the nonlinear solver will dictate to what order the symplectic structure is preserved and consequently, how well near-energy conservation is preserved (see \figurename{\ref{figalmost}}).

In practice, setting the nonlinear solver tolerance one or two orders above the order of the integrator is sufficient to maintain symplecticity. For most Taylor variational integrators, the nonlinear solver with moderate tolerance converges in a few iterations, and often in one or no iterations. The symmetric Taylor variational integrator showed excellent nonlinear convergence, and only required one iteration of the nonlinear solver for the various experiments we ran. 

\begin{figure}[h]
  \includegraphics[width=0.9\textwidth]{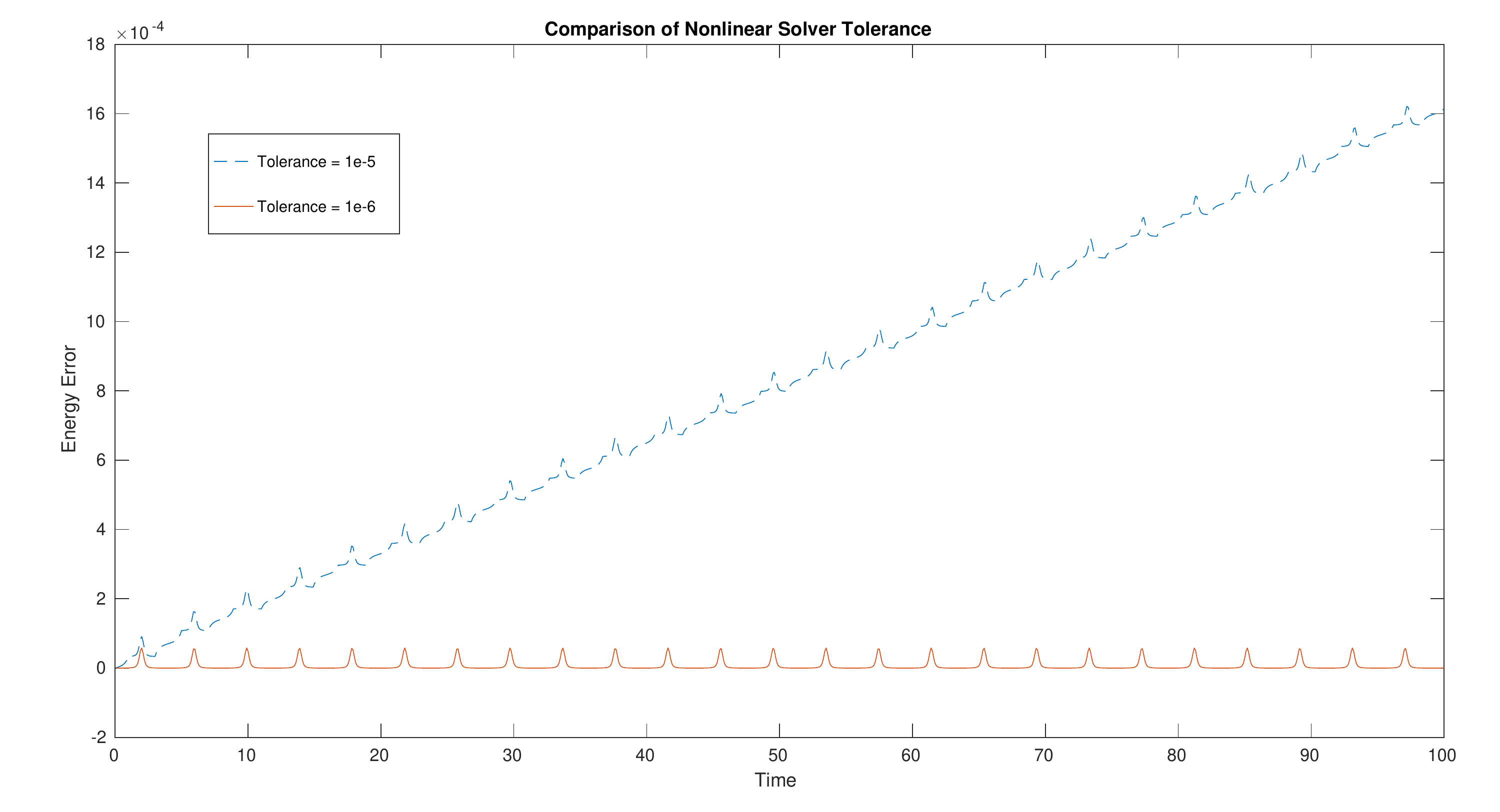}
  \caption
   {The plot of the energy preservation of a 4th order Taylor variational integrator applied to the simple pendulum with two different tolerance levels for the nonlinear solver and a step size of 0.1. Energy drift is evident when the tolerance level is set at $10^{-5}$ or larger, but the drift disappears for smaller tolerance levels. The method had an average energy error around $6.5 \cdot 10^{-5}$ for a tolerance of $10^{-6}$, and an average energy error of $8.1 \cdot 10^{-4}$ for a tolerance of $10^{-5}$.}
   \label{figalmost}
\end{figure}

 \subsection{Automatic Differentiation}
As with the Taylor method, an efficient general purpose implementation will require an efficient means of computing derivatives, such as automatic differentiation. For the following simulations, we used the AdiGator automatic differentiation package for MATLAB (see \cite{PaWeRa2013}). Implementation of a high-order Taylor variational integrator requires both the evaluation of higher time derivatives, $q^{(p+1)}(q_0,\tilde{v}_0)$, and the evaluation of the Jacobians of the time derivatives w.r.t. $q_0$ and $\tilde{v}_0$. The Jacobian evaluations are the most expensive part of the method (see \figurename{\ref{figderiv}}), especially for higher-dimensional systems, and for efficient high-order methods, the cost of Jacobian evaluations will need to be reduced to a level comparable to the time derivative. There appears to be some relationships between the Jacobians and the time derivatives that could potentially be exploited to decrease the evaluation costs. For instance,
\begin{equation*}
\frac{\partial q^{(3)}(q_0,\tilde{v}_0)}{\partial \tilde{v}_0} = \left[ q^{(3)}\bigg(q_0,\begin{bmatrix} 1 \\ 0 \end{bmatrix}\bigg) \ \ q^{(3)}\bigg(q_0,\begin{bmatrix} 0 \\ 1 \end{bmatrix}\bigg) \right],
\end{equation*}
which allows us to replace expensive Jacobian evaluations with cheaper time derivative evaluations. Additionally, Jacobians of higher-order time derivatives appear to have some relations to Jacobians of lower-order time derivatives, such as,
\begin{equation*}
\frac{\partial q^{(4)}(q_0,\tilde{v}_0)}{\partial \tilde{v}_0} = -2\frac{\partial q^{(3)}(q_0,\tilde{v}_0)}{\partial q_0}.
\end{equation*}
Hopefully, a good implementation of automatic differentiation will already take advantage of such relationships.

Automatic differentiation greatly benefits from the way it is compiled, which means the more efficient implementations will be in languages such as Fortran or C++. Another aspect to consider is parallel implementation. Combining automatic differentiation and parallel computing techniques has been shown to significantly reduce computational time (see \cite{BuLaMeBi01}).

One possible implementation for the algorithm would be to construct the Taylor discrete Lagrangian, then apply automatic differentiation to the discrete Lagrangian in combination with a nonlinear solver to recover the discrete Legendre transforms and consequently $(q_1,p_1)$. In fact, this could provide a more general framework for the derivation of all implicit variational integrators.

\begin{figure}[h]
  \includegraphics[width=0.9\textwidth]{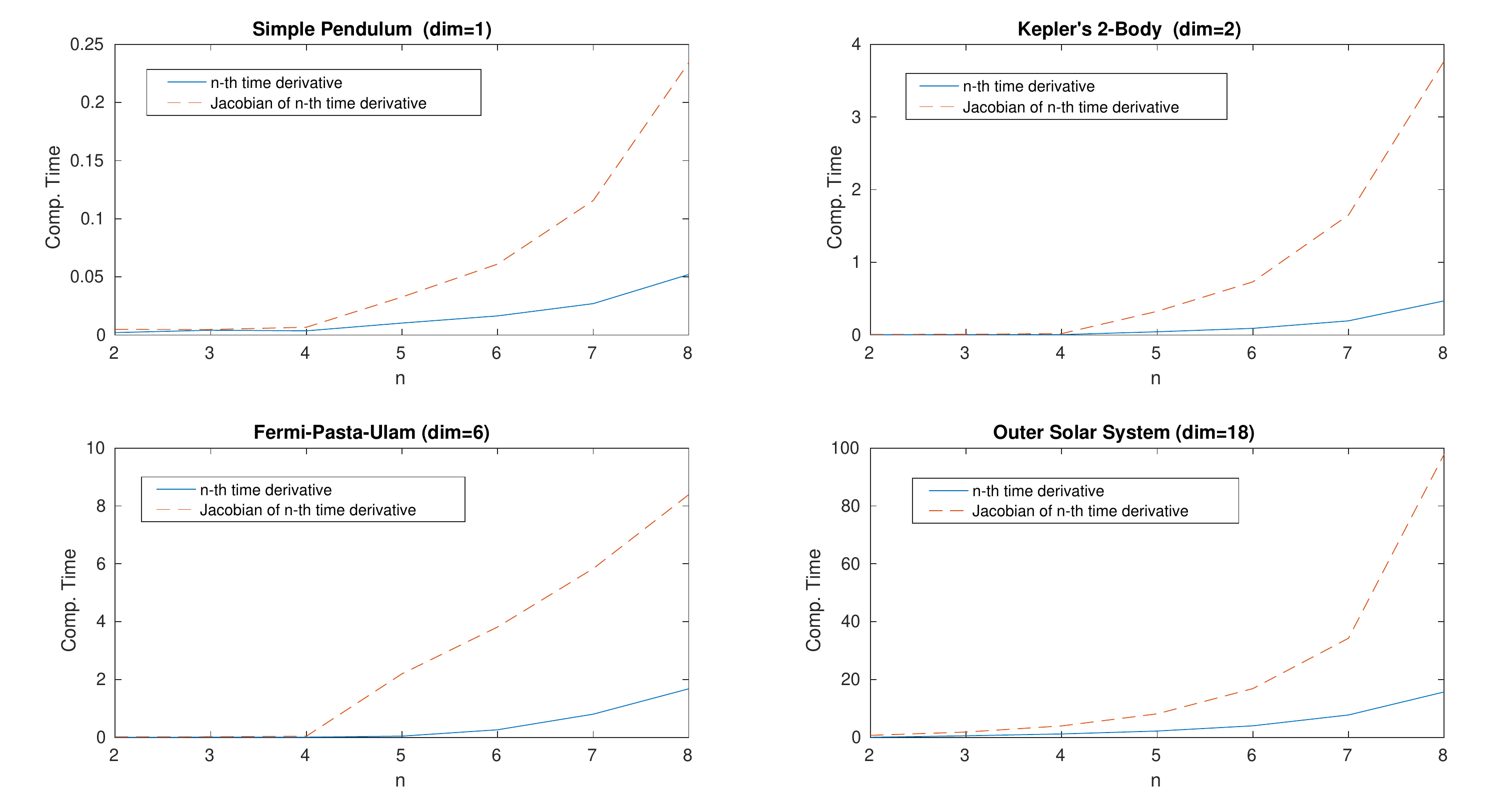}
  \caption
   {The derivative order versus time plot of 100 evaluations of each derivative corresponding to 4 different models with increasing dimension. It is worth noting that the rate of growth in time needed for higher-order derivative evaluations appears to be independent of the dimension.}
   \label{figderiv}
\end{figure}

\subsection{Comparison of Methods}
  The simulations compare the discrete Lagrangian form of the Taylor variational integrator (TVI), the discrete right Hamiltonian form of the Taylor variational integrator (HTVI), the symmetric Taylor variational integrator of 4th order (SV4), Taylor's method, and the Runge--Kutta shooting variational integrators (ShVI) (see \cite{LeSh2011_sbvi}). Overall, high-order Taylor methods perform quite well in terms of computational time versus global error. However, as the length of integration time becomes very large, the variational integrators begin to show their strength. Of the three variational integrators, the symmetric Taylor variational integrator is the most efficient. 

Comparison of the Lagrangian or Hamiltonian Taylor variational integrator to the Runge--Kutta shooting variational integrator does not result in a clear winner in terms of computational efficiency. It is well known that beyond 4th-order, Runge--Kutta (RK) methods require a higher number of stages/function evaluations, and the number of stages grows faster for vector differential equations as compared to scalar differential equations (see \cite{Bu2009}). The number of order conditions grows quite quickly. For instance a 4th-order RK method has 8 order conditions, a 7th-order RK method has 85 order conditions, and a 25th-order method has 3,231,706,871 order conditions (see \cite{Somm98}). However, a 25th-order RK method only has 313 stages, so the function evaluations grow at a much slower rate. The Taylor method must contend with the increasing cost of evaluating higher-order derivatives, which for our implementation grows at a rate of $2^n$, where $n$ is the order of the derivative. For methods less than order 10 the difference in computational cost of the Taylor variational integrator and the Runge--Kutta based shooting variational integrator did not seem significant. However, the symmetric Taylor variational integrator did exhibit lower evaluation costs than the other methods. It should be noted that the most efficient implementations of the Taylor method involve variable stepsizes, and symplectic integrators are not predisposed to variable stepsizes.

The following simulations were implemented in MATLAB.
\subsection{Simple Pendulum}
Consider the simple pendulum with unit mass and length in a gravitational field with $g=-9.8\mathrm{m/s}^2$, where $q$ is parametrized by the angle between the y-axis and the pendulum. The corresponding Lagrangian is,
$$L(q,\dot{q})=\frac{1}{2} \dot{q}^2-g(1-\cos(q)).$$
The Euler--Lagrange equation yields,
$$\ddot{q}=-g\sin(q).$$

In \figurename{\ref{figpen}}, the level sets of the corresponding Hamiltonian are compared to the trajectories generated by a 2nd-order Taylor variational integrator (TVI2) (see Example 2). The numerical solutions appear nearly identical to the level sets of the Hamiltonian, which indicates that the variational integrator exhibited good energy behavior for a variety of initial conditions.

The simulation in \figurename{\ref{figPen6}} used initial conditions $(q_0,p_0)=(\frac{\pi}{2},0)$. The 6th-order Taylor variational integrator performed well at a stepsize of $h=0.5$, while the 6th-order Taylor method failed to generate a reasonable approximation for this stepsize. The ability of the Taylor variational integrator to perform well at larger stepsizes may gives it an advantage over traditional Taylor methods.

In \figurename{\ref{figtiming}}, we compare various types of Taylor variational integrators against the shooting-based variational integrator (ShVI). The plots compare the energy error versus computational time for methods of various order. It is clear the the symmetric Taylor variational integrator (SV4) is the most efficient in this respect, but it is not so clear whether the non-symmetric Taylor variational integrators (TVI and HTVI) are more efficient than ShVI.

\begin{figure}
\centering
\begin{subfigure}{.5\textwidth}
  \centering
  \includegraphics[width=.8\linewidth]{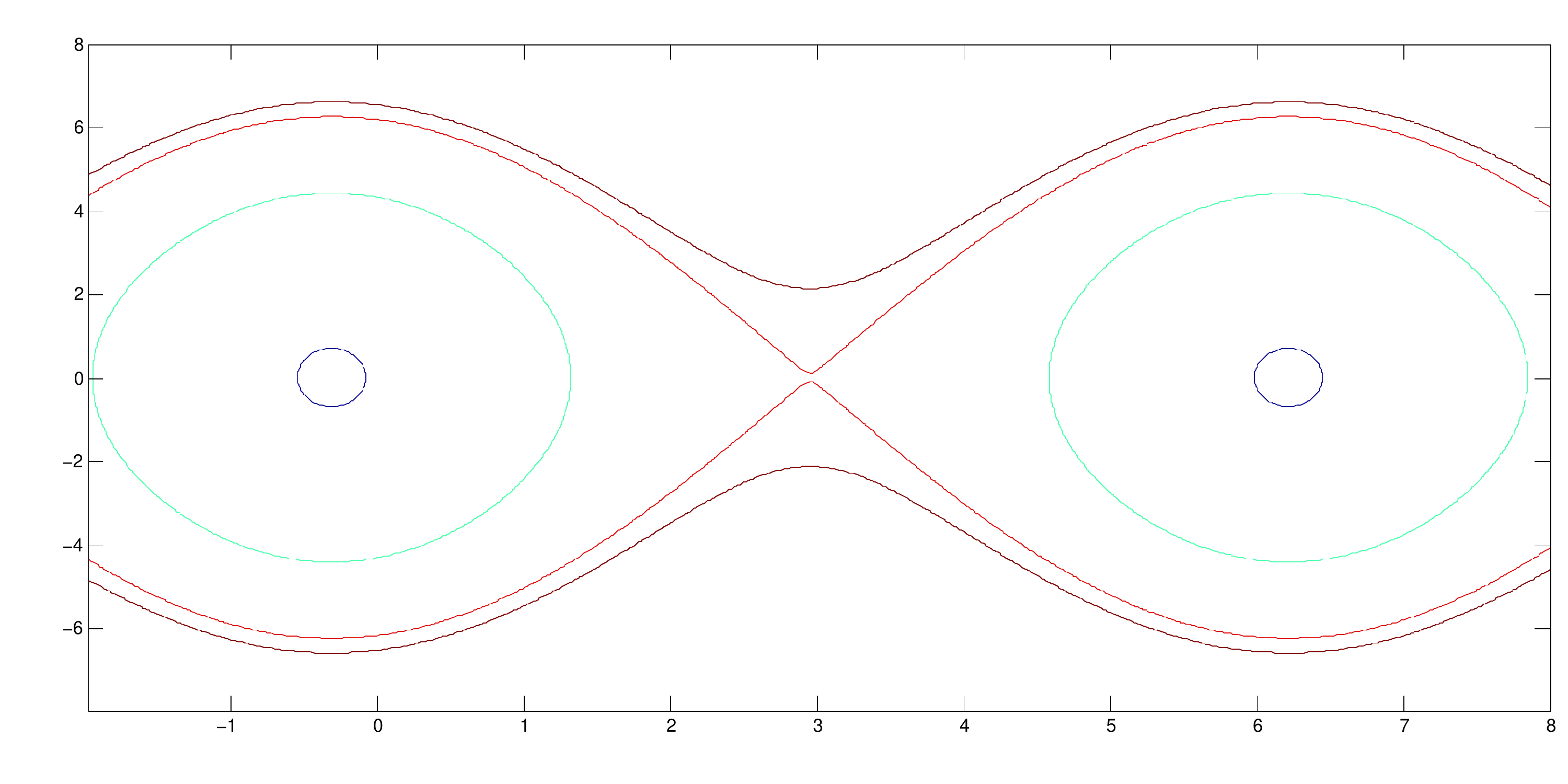}
  \caption{}
\end{subfigure}%
\begin{subfigure}{.5\textwidth}
  \centering
  \includegraphics[width=.8\linewidth]{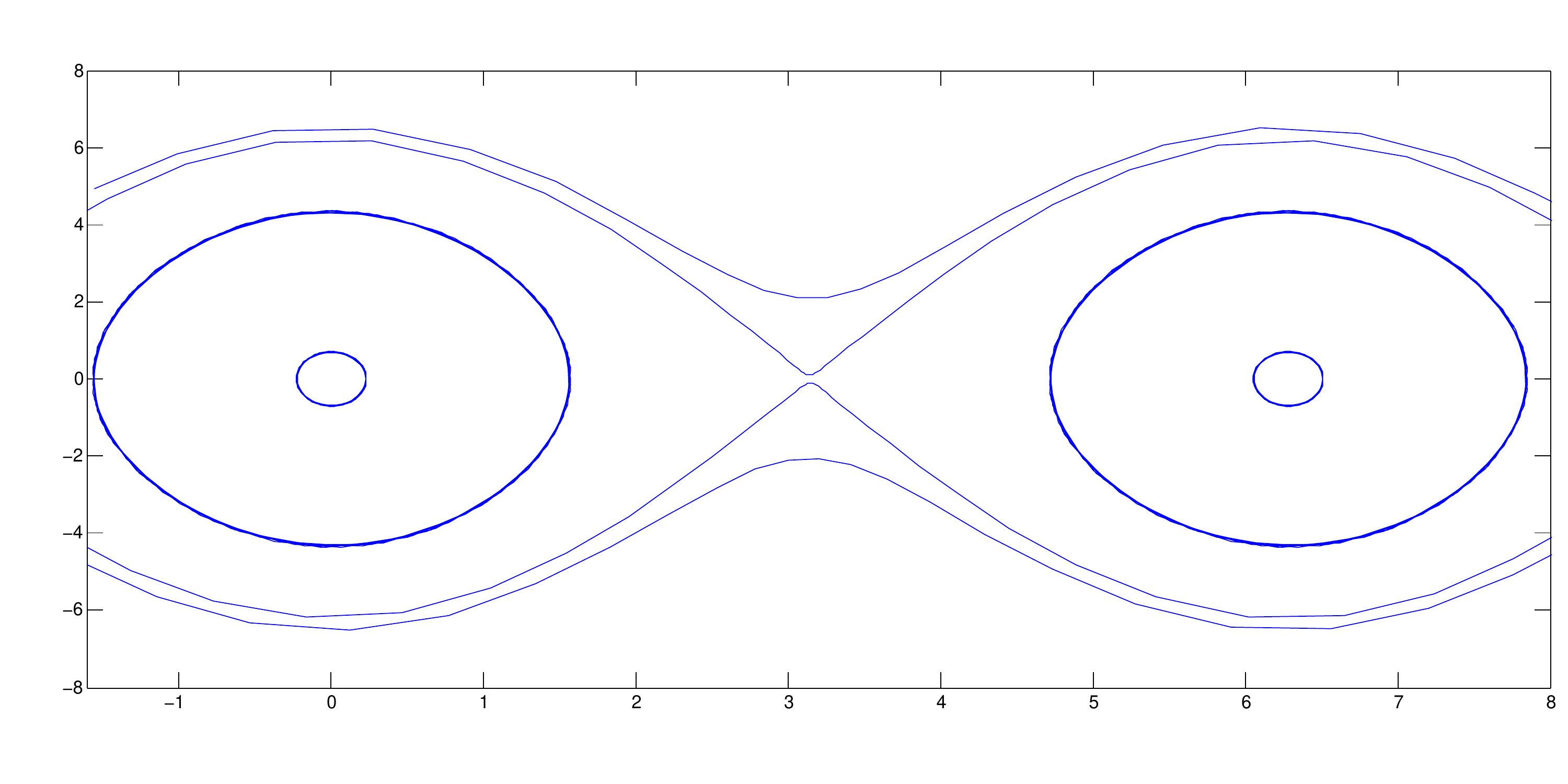}
  \caption{}
\end{subfigure}
\caption{(A) The level sets of the Hamiltonian of the simple pendulum corresponding to a variety of initial conditions. (B) The trajectories generated by TVI2 using the same initial conditions with a step size $h=0.1$ for the time interval $[0,20]$.}
\label{figpen}
\end{figure}

\begin{figure}[h]
  \includegraphics[width=0.9\textwidth]{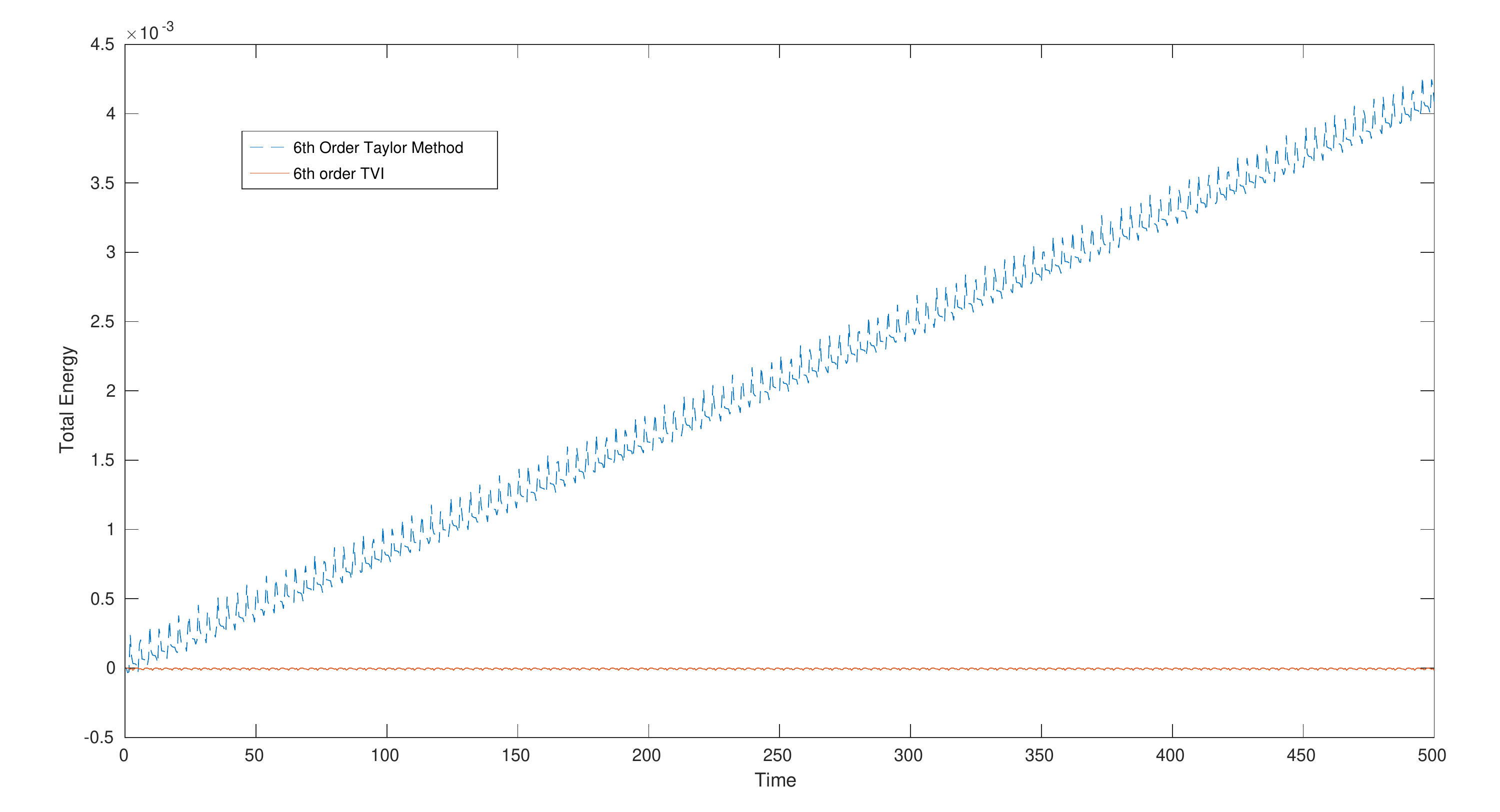}
  \caption
   {A plot of the Simple Pendulum total energy vs. time of the sixth-order integrators TVI6 and Taylor's method for a step size of $h=0.5$. At this step size and time interval, Taylor's method has significant energy drift, and as a result its accuracy suffers.}
   \label{figPen6}
\end{figure}

\subsection{Kepler's Planar $2$-Body Problem}
Consider two bodies interacting under mutual gravity and set one body as the center of the coordinate system (see \cite{HaLuWa2006}). Thus, constraining them to lie in a plane, we have Kepler's planar 2-body problem with corresponding Lagrangian,
$$L(\mathbf{q},\mathbf{\dot{q}})=\frac{1}{2}(\dot{q}_1^2+\dot{q}_2^2)+(q_1^2+q_2^2)^{-1/2}.$$
Note here we are using $q_1$ and $q_2$ as the first and second components of $\mathbf{q}$. This in turn gives us the Euler--Lagrange equations,
$$\mathbf{\ddot{q}}=\begin{bmatrix} \frac{-q_1}{(q_1^2+q_2^2)^{3/2}} \\ \frac{-q_2}{(q_1^2+q_2^2)^{3/2}} \end{bmatrix}.$$

Our simulations used initial conditions $\mathbf{q}_0=\begin{bmatrix} 1 \\ 0 \end{bmatrix}$ and $\mathbf{p}_0=\begin{bmatrix} 0 \\ 0.8 \end{bmatrix}$. \figurename{\ref{figkep25}} compares various Taylor variational integrators to Taylor methods of the same order using a stepsize of $h=0.25$. The trajectories of the Taylor methods for this stepsize behave poorly, while variational integrators show good qualitative performance.

\figurename{\ref{SVHd}} compares the St\"ormer--Verlet method (SV) to the discrete Hamiltonian composition method (SVHd) discussed in section \ref{discHam}. Given that the St\"ormer--Verlet method is explicit, while SVHd is implicit, it is no surprise that the St\"ormer--Verlet method has lower computational cost. However, SVHd does exhibit lower energy error and performs slightly better qualitatively, so when the problem is non-separable (and SV is implicit), SVHd may be a better alternative. 

\begin{figure}[h]
  \includegraphics[width=0.9\textwidth]{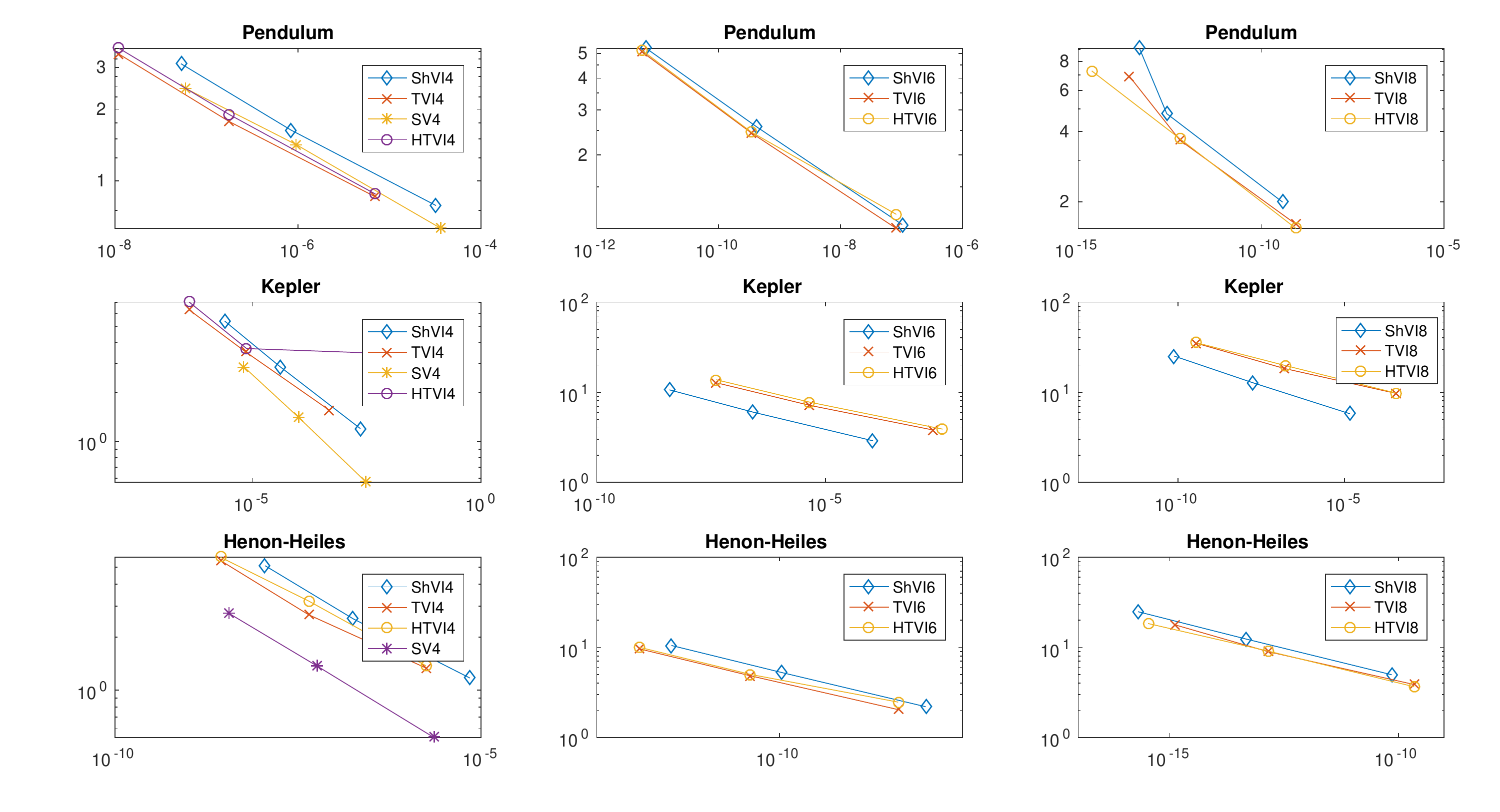}
  \caption
   {Plots of the average energy error versus computational time for the various variational integrators. The 4th-order symmetric Taylor variational integrator (SV4) is the clear winner in terms of efficiency, while comparisons of TVI, HTVI, and ShVI are mixed.}
   \label{figtiming}
\end{figure}

\begin{figure}[h]
  \includegraphics[width=0.9\textwidth]{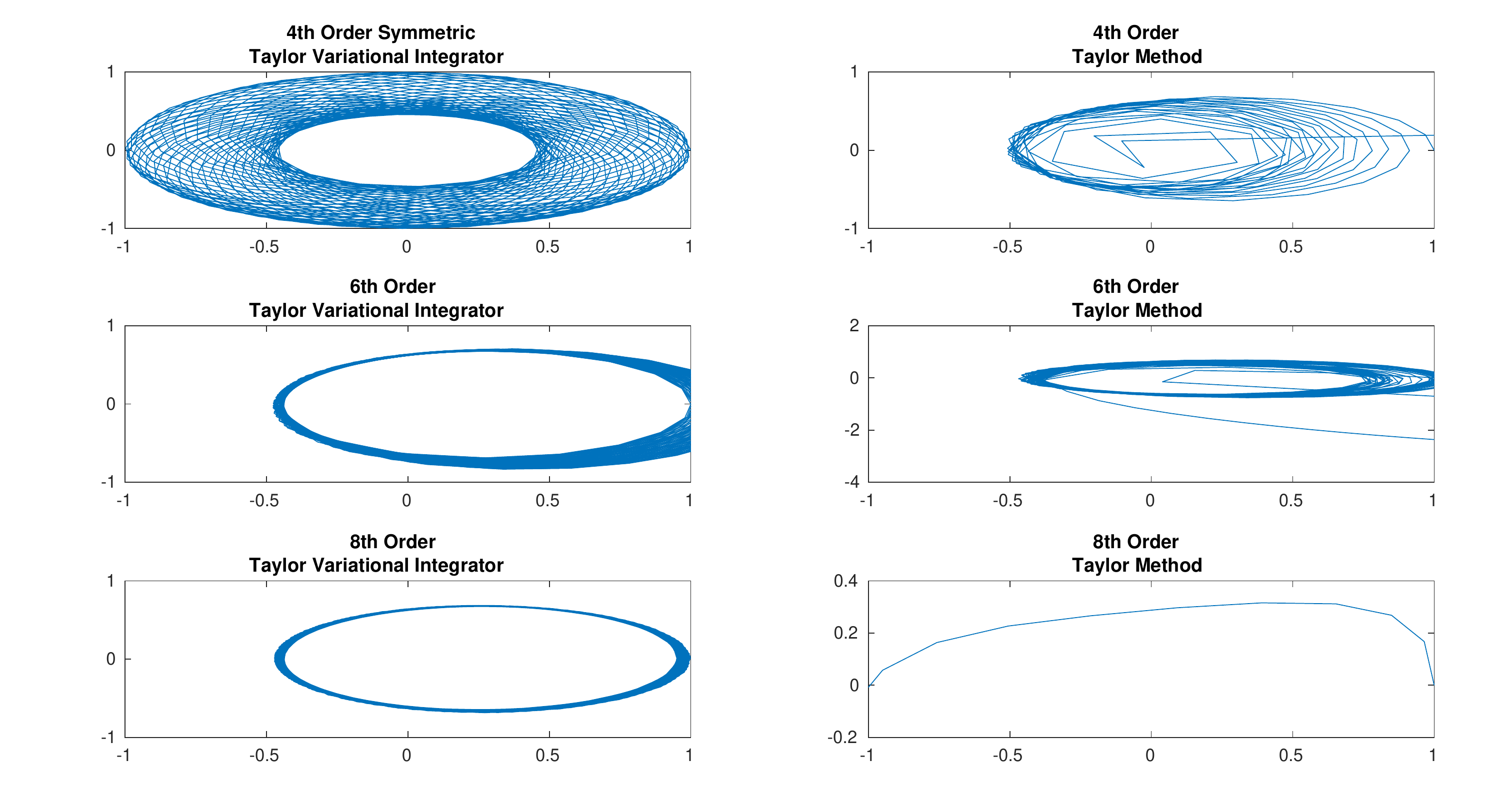}
  \caption
   {Position plots of Kepler's planar 2-body problem as generated by various integrators with a time step of $h=0.25$ over a time interval of $[0,250]$. The Taylor variational integrators exhibit close to the correct behavior, while the various Taylor methods all fail to capture the behavior of the system.}
   \label{figkep25}
\end{figure}

\begin{figure}
  \includegraphics[width=0.9\textwidth]{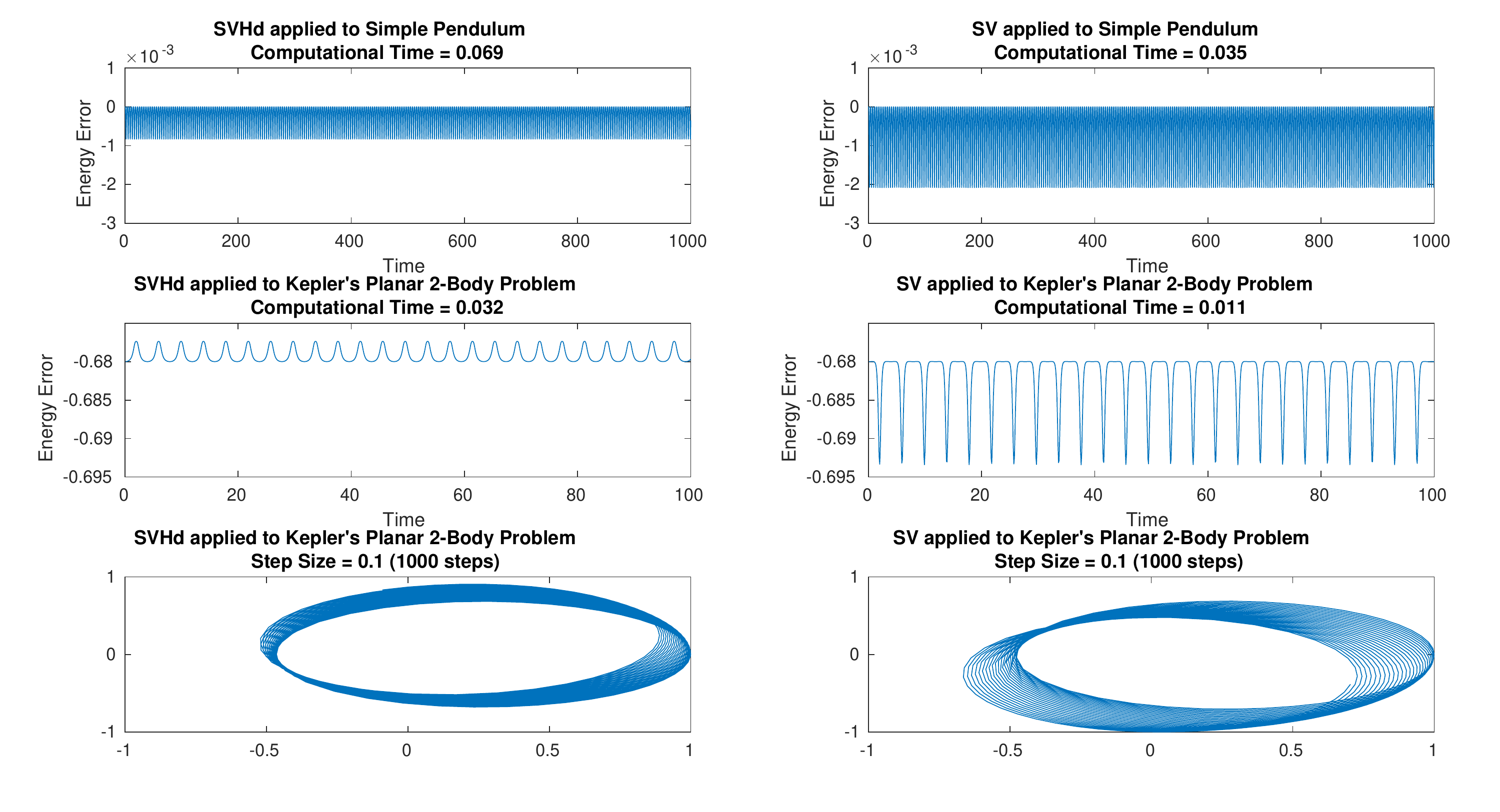}
  \caption
   {This plot compares the performance of St\"ormer--Verlet (SV) and the discrete Hamiltonian composition method (SVHd) from section \ref{discHam}. SVHd exhibits a much smaller amplitude in the energy error, as compared to SV, but the implicit nature of SVHd is reflected in the increased computational cost. Clearly, SV is preferable for separable problems, but for non-separable problems SVHd may be the better choice.}
   \label{SVHd}
\end{figure}

\begin{figure}[h]
  \includegraphics[width=0.9\textwidth]{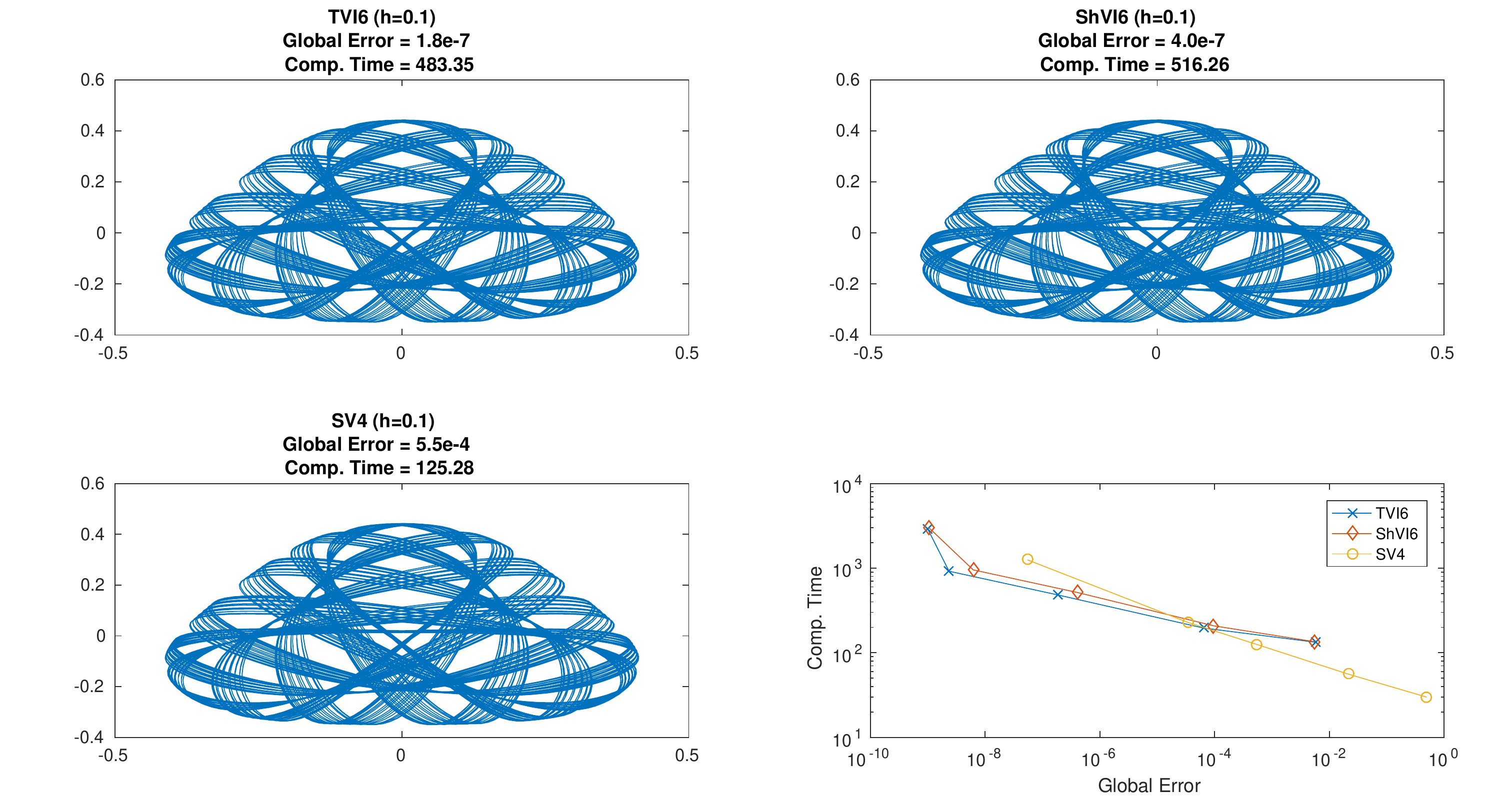}
  \caption
   {The Henon-Heiles model simulated over the time interval $[0,1000]$. The bottom right plot compares the global error versus computational time of the 6th-order Taylor variational integrator (TVI6), the 6th-order Runge--Kutta based shooting variational integrator (ShVI6), and the 4th-order symmetric Taylor method (SV4).}
   \label{fighenon}
\end{figure}

\begin{figure}
  \includegraphics[width=0.9\textwidth]{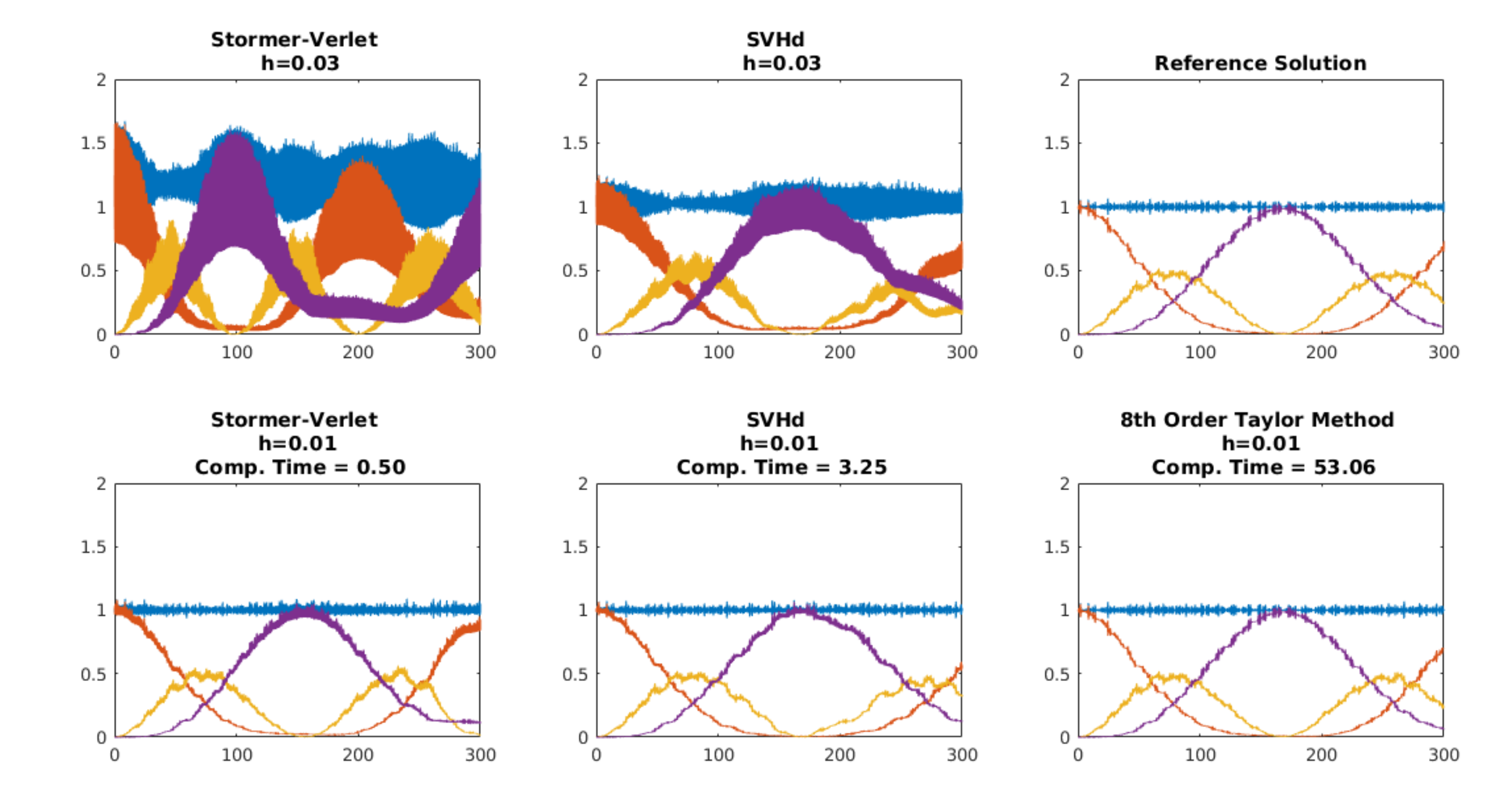}
  \caption
   {A comparison of St\"ormer--Verlet (SVHd) and the 8th-order Taylor method for the Fermi-Pasta-Ulam model. For $h=0.03$, the St\"ormer--Verlet method is on the cusp of being linearly unstable. For $h=0.01$, the methods all present a similar picture to the reference solution, but their global errors are quite large and none of them exhibit good accuracy. }
   \label{figFPU}
\end{figure}

\begin{figure}[h]
  \includegraphics[width=0.9\textwidth]{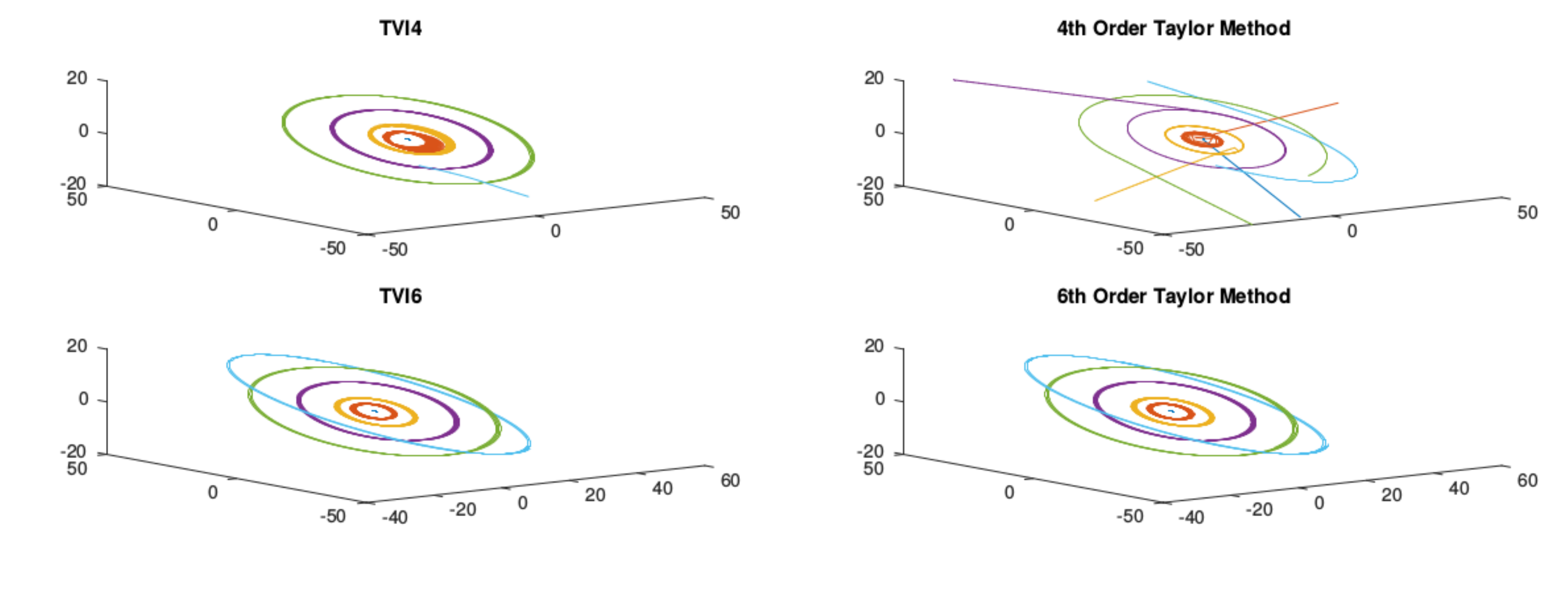}
  \caption
   {The sun and 5 outer planets simulated over the time interval [0,200000] with a step size of $h=400$ (days). The stepsize is too large for the 4th-order methods to give a qualitatively accurate representation, but both 6th-order methods performed well qualitatively.}
   \label{figout1}
\end{figure}

\subsection{Henon-Heiles Model}
 The Henon--Heiles model attempts to capture the dynamics of a galaxy with cylindrical symmetry (see \cite{HaLuWa2006} for more info). The Hamiltonian is given by, $H(p,q)=\frac{1}{2}(p_1^2+p_2^2)+U(q)$, where $U(q)=\frac{1}{2}(q_1^2+q_2^2)+q_1^2q_2-\frac{1}{3}q_2^3$. The corresponding Euler--Lagrange equation is,
 \begin{align*}
 \mathbf{\ddot{q}} = \begin{bmatrix} -q_1-2q_1q_2 \\ -q_2-q_1^2+q_2^2 \end{bmatrix}.
 \end{align*}
 It is known that the dynamics become chaotic at higher energy levels. The following simulations were conducted with an initial energy level of $H_0=\frac{1}{12}$ (see \figurename{\ref{fighenon}}) and $H_0=\frac{1}{8}$ (see \figurename{\ref{figtiming}}). The second energy value corresponds to a chaotic system.
 
 In \figurename{\ref{fighenon}}, we compare the 6th-order Taylor variational integrator (TVI6), the 6th-order Runge--Kutta shooting-based variational integrator (ShVI6), and the 4th-order symmetric Taylor variational integrator (SV4) applied to the Henon-Heiles model with $H_0=\frac{1}{12}$. For global errors between $10^{-1}$ and $10^{-5}$, SV4 is the more efficient method. Amongst the higher-order methods, TVI6 and ShVI6 appear to be the more efficient methods. A 6th-order symmetric Taylor variational integrator would be even more efficient for higher-order accuracy.
 
\subsection{Fermi-Pasta-Ulam Model}
  The Fermi-Pasta-Ulam (FPU) model has a particularly distinguished place in the history of numerical simulations and nonlinear dynamics (see \cite{Ford1992}). We apply the modified model as outlined in \cite{HaLuWa2006}, consisting of a sequence of 6 mass points, fixed at both ends connected on opposite sides by a series of soft nonlinear springs and stiff linear springs. Letting $\{q_i,p_i\}_{i=1}^6$ denote the displacements and velocities of the mass points, the corresponding Hamiltonian is given by,
\begin{equation*}
H(p,q) = \frac{1}{2} \sum^3_{i=1} (p^2_{2i-1}+p^2_{2i}) + \frac{\omega^2}{4} \sum_{i=1}^6 (q_{2i}-q_{2i-1})^2 + \sum_{i=0}^6 (q_{2i+1}-q_{2i})^4,
\end{equation*}
where $\omega=50$. By using the change of variables,
\begin{align*}
x_{0,i} &= (q_{2i}+q_{2i-1})/\sqrt{2}, \qquad x_{1,i}=(q_{2i}-q_{2i-1})/\sqrt{2}, \\
y_{0,i} &= (p_{2i}+p_{2i-1})/\sqrt{2}, \qquad y_{1,i}=(p_{2i}-p_{2i-1})/\sqrt{2},
\end{align*}
the resulting Hamiltonian system has a nearly conserved quantity $I=I_1+\dots+I_3$, where
\begin{equation*}
I_j(x_{1,j},y_{1,j})=\frac{1}{2}(y^2_{1,j}+\omega^2x^2_{1,j})
\end{equation*}
is the energy of the $j$th stiff spring. Despite the significant energy exchange between individual springs, the total oscillatory energy, $I$, remains near constant. Our simulations used initial values of,
\begin{equation*}
\begin{bmatrix} x_{0,1} \\ x_{0,2} \\ x_{0,3} \\ x_{1,1} \\ x_{1,2} \\ x_{1,3} \end{bmatrix} = \begin{bmatrix} 1 \\ 0 \\ 0 \\ 1/\omega \\ 0 \\ 0 \end{bmatrix}, \qquad \begin{bmatrix} y_{0,1} \\ y_{0,2} \\ y_{0,3} \\ y_{1,1} \\ y_{1,2} \\ y_{1,3} \end{bmatrix} = \begin{bmatrix} 1 \\ 0 \\ 0 \\ 1 \\ 0 \\ 0 \end{bmatrix}. 
\end{equation*}

\figurename{\ref{figFPU}} compares the St\"ormer--Verlet method to SVHd. The first couple of plots use a stepsize of $h=0.03$, which is on the boundary of the linear stability of St\"ormer--Verlet (i.e. $h\omega=1.5$). SVHd does appear to be qualitatively more accurate, but neither method does well at this stepsize. For $h=0.01$, both methods give a much better qualitative representation of the system, but their global errors are still too large to be considered accurate. None of the methods in this paper are appropriate for a highly-oscillatory model such as the FPU model. For an accurate solution, one should consider either the IMEX method (see \cite{StGr2009}) or Filon-type methods (see \cite{IsNor2000}). The combination of exponential type integrators with symplectic and energy-preserving integrators was also recently considered in \cite{ShLe2017}.

\subsection{Outer Solar System}
Consider the motion of the five outer planets (including Pluto) relative to the sun. The corresponding Hamiltonian for this N-body problem is given by,
\begin{equation*}
H(p,q) = \frac{1}{2} \sum^5_{i=0} \frac{1}{m_i} p_i^Tp_i - G \sum^5_{i=1} \sum^{i-1}_{j=0} \frac{m_im_j}{\|q_i-q_j\|},
\end{equation*}
where $G=2.95912208286 \cdot 10^{-4}$. The initial data and masses is taken from Section 1.2.4 of \cite{HaLuWa2006}, and corresponds to September 5, 1994 at 0h00. In \figurename{\ref{figout1}}, we compare the 4th and 6th-order Taylor variational integrators to the 4th and 6th-order Taylor methods. The simulations was over the time period $[0,200000]$, and the stepsize was $h=400$ (days). The 4th-order methods did not produce a useful simulation at this stepsize, but both 6th-order integrators give a good representation of the system.

\section{Conclusions and Future Directions}
The Taylor variational integrators provide a way to build high-order symplectic integrators and include many of the classic symplectic integrators as special cases, i.e., symplectic Euler and St\"ormer--Verlet. This provides a framework for importing the large body of literature on the efficient construction of high-order Taylor integrators in order to construct similarly high-order symplectic integrators.

In particular, these methods can be viewed as a symplectic correction to higher-order Taylor methods that typically converges in a small number of iterations. By viewing these as predictor-corrector methods, one can interpolate between Taylor methods and Taylor variational integrators, and it would be interesting to see the extent to which a fixed number of iterations of the symplectic corrector can improve upon the performance of Taylor integrators for realistic problems.

The numerical simulations demonstrate that the geometric structure-preserving properties of symplectic integrators can be important for achieving numerical stability of long time simulations, so it should be of great interest to the computational astrophysics community to combine the high-order accuracy of high-order Taylor integrators with the geometric structure-preserving properties of variational integrators.

The most efficient implementations of the Taylor method utilize a variable stepsize, and extending variable stepsizes to the variational integrator framework is an area that deserves continued research. We are currently considering an approach based on the combination of Hamiltonian variational integrators and the Poincar\'e transformation that is quite promising. In particular, we note that the use of Hamiltonian as opposed to Lagrangian variational integrators is critical, as the Poincar\'e transformed Hamiltonian is degenerate, and there is no corresponding Lagrangian formulation.

%

\section*{Acknowledgements} The authors were supported in part by NSF grants DMS-1001521, CMMI-1029445, DMS-1065972, CMMI-1334759, DMS-1411792, and NSF CAREER award DMS-1010687.

\appendix
\section{Detailed Proofs}
Given an Euler--Lagrange equation of the form,
$$\ddot{q}(t)=f(q(t),\dot{q}(t),t),$$
we denote the exact solution of the Euler--Lagrange boundary-value problem with boundary conditions $(q_0,q_1)$ by $(q(t),v(t))$. We seek an estimate of the true initial velocity, ${v}_0$, for the corresponding Euler--Lagrange initial-value problem, with order of accuracy $r$. Let us denote this estimate by $\tilde{v}_0$. Given a one-step method, $\hat{\Psi}_h:TQ\rightarrow TQ$, with order of accuracy $r+1$, we solve for the initial velocity $\tilde{v}_0$, such that,
\begin{equation}
\pi_Q\circ\hat{\Psi}_h(q_0,\tilde{v}_0)=q_1,\label{tildev0}
\end{equation}
where $\pi_Q:TQ\rightarrow Q$ is the canonical projection. Let $\Phi_h:TQ\rightarrow TQ$ be the exact time-$h$ flow map of the Euler--Lagrange initial-value problem. By definition, the exact Euler--Lagrange flow applied to the initial condition $(q_0,v_0)$ is a solution of the Euler--Lagrange boundary-value problem with boundary conditions $(q_0,q_1)$, where
\begin{equation}
\pi_Q\circ\Phi_h(q_0,v_0)=q_1.\label{q1}
\end{equation}

Consider a Taylor method with order of accuracy $r$ and $r+1$,
\begin{equation}
\Psi_h(q_0,\tilde{v}_0)=\left(\sum\nolimits_{k=0}^r \frac{h^k}{k!}q^{(k)}(0),\sum\nolimits_{k=1}^{r+1} \frac{h^{k-1}}{(k-1)!}q^{(k)}(0)\right)\label{Psi}
\end{equation}
and
\begin{equation}
\hat{\Psi}_h(q_0,\tilde{v}_0)=\left(\sum\nolimits_{k=0}^{r+1} \frac{h^k}{k!}q^{(k)}(0),\sum\nolimits_{k=1}^{r+2} \frac{h^{k-1}}{(k-1)!}q^{(k)}(0)\right),\label{Psihat}
\end{equation}
where $q^{(k)}(0)$ is calculated by considering the prolongations of the Euler--Lagrange vector field, and evaluating it at $(q_0,\tilde{v}_0)$. An analogous approach, involving the prolongation of the Euler--Lagrange vector field at both the initial and final time, which can be viewed as a two-point Taylor method, was used to develop a prolongation-collocation variational integrator in \cite{LeSh2011}.
\begin{lemma} 
$\tilde{v}_0$ as defined by, \eqref{tildev0} and \eqref{Psihat}, approximates $v_0$ to at least $\mathcal{O}(h^{r+1})$.
\end{lemma}
\begin{proof}
Solving $\pi_Q\circ\hat\Psi_h(q_0,\tilde{v}_0)=q_1$ for $\tilde{v}_0$ yields,

$$\tilde{v}_0=\frac{q_1-q_0}{h}-\sum\nolimits_{k=1}^{r} \frac{h^k}{(k+1)!}(f^{(k-1)}(q_0,\tilde{v}_0,0)).$$
Since the exact solution $q(t) \in C^{r+2}([0,h])$, using Taylor's Theorem, we have,

$$q_1=q_0+v_0h+\sum\nolimits_{k=2}^{r+1} \frac{h^k}{k!}f^{(k-2)}(q_0,v_0,0)+R_{r+1}(h).$$
Solving for $v_0$ yields,

$$v_0=\frac{q_1-q_0}{h}-\sum\nolimits_{k=1}^{r} \frac{h^k}{(k+1)!}f^{(k-1)}(q_0,v_0,0)-\frac{R_{r+1}(h)}{h}.$$
Now evaluating the norm of the difference we have,
$$\|\tilde{v}_0-v_0\|=\left\| -\sum\nolimits_{k=1}^{r} \frac{h^k}{(k+1)!}(f^{(k-1)}(q_0,\tilde{v}_0,0)-f^{(k-1)}(q_0,v_0,0))+\frac{R_{r+1}(h)}{h}\right\|.$$
Since $q(t) \in C^{r+2}([0,h])$ each of $f^{(i-1)}$ is Lipschitz continuous in its arguments for $i=1,2,\ldots,r$. Let $M_i$ be the Lipschitz constant for $f^{(i-1)}$ over the compact interval $[0,C]$ with respect to velocity, and $C>0$ can be chosen so that $M_i$, $i=1,2,\ldots,p,$ is bounded.  Using the triangle inequality, we have,
\[
\|\tilde{v}_0-v_0\| \leq \sum\nolimits_{k=1}^{r}\frac{h^k}{(k+1)!}M_k \|\tilde{v}_0-v_0\|+\left\|\frac{R_{r+1}(h)}{h}\right\|.
\]
Rearranging, we have,
\[
\|\tilde{v}_0-v_0\|\left(1-\sum\nolimits_{k=1}^r \frac{h^k}{(k+1)!}M_k\right) \leq \left\|\frac{R_{r+1}(h)}{h}\right\|\leq \frac{\mathcal{O}(h^{r+2})}{h}=\mathcal{O}(h^{r+1}).
\]
By continuity, there exists $\tilde{C}$ satisfying $0<\tilde{C}<C$, such that for all $h$ satisfying $0<h<\tilde{C}$, the term inside the parenthesis on the leftmost expression is positive and bounded away from zero. That concludes the proof.\end{proof}

\begin{remark*}
It is worth noting that a similar proof may be given for any $(r+1)$-order one-step method. This is due to the fact that any $(r+1)$-order one-step method agrees with the $(r+1)$-order Taylor's method up to a local truncation error of order $\mathcal{O}(h^{r+2})$. Thus, the only change in the error term in the proof would be to replace $R_{r+1}(h)$ by the sum of the local truncation error of the one-step method and $R_{r+1}(h)$, which are both $\mathcal{O}(h^{r+2})$. Thus, this result can be generalized to any one-step method of the desired order.
\end{remark*}

Using this result, we can show that starting our $r$-order Taylor method at $\tilde{v_0}$, rather than at $v_0$, will not affect the order of accuracy of the method.

\begin{lemma}
A $r$-order Taylor method, defined by \eqref{Psi}, with initial conditions $(q_0,\tilde{v}_0)$, where $\tilde{v}_0$ is defined by \eqref{tildev0}, is accurate to at least $\mathcal{O}(h^{r+1})$ for the Euler--Lagrange boundary-value problem with boundary conditions $(q_0,q_1)$.
\end{lemma}
\begin{proof}
As before, we denote the solution to the Euler--Lagrange boundary-value problem with boundary condition $(q_0,q_1)$ by $(q(t),v(t))$ for $t\in[0,h]$. This solution also satisfies the Euler--Lagrange initial-value problem with initial conditions $(q_0,v_0)$, where $v_0$ satisfies \eqref{q1}. We denote the solution of the Euler--Lagrange initial-value problem with initial conditions $(q_0,\tilde{v}_0)$ by $(\tilde{q}(t),\tilde{v}(t))$. Let $(q_d(t),v_d(t))$ denote the values generated by $r$-order Taylor method with initial conditions $(q_0,\tilde{v}_0)$. Noting that the Euler--Lagrange initial-value problem is well-posed, we denote the Lipschitz constant with respect to initial velocity by $M$.
\begin{align*}
\|(q(t),v(t))-(\tilde{q}(t),\tilde{v}(t))\|&\leq M\|v_0-\tilde{v}_0\| \leq \mathcal{O}(h^{r+1}).
\end{align*}
Combining this inequality with our $r$-order method yields,
\begin{align*}
\|(q(t),v(t))-(q_d(t),v_d(t))\| &= \|(q(t),v(t))-(\tilde{q}(t),\tilde{v}(t))+(\tilde{q}(t),\tilde{v}(t))-(q_d(t),v_d(t))\|\\
&\leq \|(q(t),v(t))-(\tilde{q}(t),\tilde{v}(t))\| + \|(\tilde{q}(t),\tilde{v}(t))-(q_d(t),v_d(t))\|\\
&\leq \mathcal{O}(h^{r+1}),
\end{align*}
where we used the triangle inequality, and the fact that the local truncation error of a $r$-order Taylor method is $\mathcal{O}(h^{r+1})$ to bound the second term in line two, since $(\tilde{q}(t),\tilde{v}(t))$ and $(q_d(t),v_d(t))$ correspond to the exact solution and $r$-th order Taylor approximation, respectively, of the Euler--Lagrange initial-value problem with initial data $(q_0,\tilde{v}_0)$.
\end{proof}

\begin{theorem} Assuming a regular Lagrangian, we consider a fixed approximation scheme used to construct a corresponding discrete Lagrangian, $L_d$, and a discrete right Hamiltonian, $H_d^+$. This results in two integrators, $\tilde{F}_{L_d}:(q_0,p_0) \mapsto (q_{1,L_d},p_{1,L_d})$ and $\tilde{F}_{H_d^+}:(q_0,p_0) \mapsto (q_{1,H_d^+},p_{1,H_d^+})$. If the discrete right Hamiltonian approximation satisfies $p_{1,H_d^+}=D_2L_d(q_0,\hat{q}_1)$, where $\hat{q}_1$ is the approximated value of $q_1$, then the integrators represent the same map, i.e., $(q_{1,L_d},p_{1,L_d}) = (q_{1,H_d^+},p_{1,H_d^+})$.
\end{theorem}

\begin{proof}
Let $\hat{p}_0$ be defined by $-\hat{p}_0=D_1L_d(q_0,\hat{q}_1)$, where we consider $\hat{q}_1$ as an independent variable. The discrete right Hamiltonian is given by,
\begin{equation*}
H_d^+(q_0,p_{1,H_d^+}) = p_{1,H_d^+}^T \hat{q}_1 - L_d(q_0,\hat{q}_1).
\end{equation*}
Note that here $\hat{q}_1$ is being considered as a function of $q_0$ and $p_{1,H_d^+}$, as defined implicitly by the assumption $p_{1,H_d^+}=D_2L_d(q_0,\hat{q}_1)$. Then
\begin{align*}
p_0 &= D_1H_d^+(q_0,p_{1,H_d^+}) \\ 
&= \frac{\partial \hat{q}_1}{\partial q_0}^T p_{1,H_d^+} - \left( D_1L_d(q_0, \hat{q}_1) + \frac{\partial \hat{q}_1}{\partial q_0}^T D_2L_d(q_0, \hat{q}_1) \right) \\
&= \frac{\partial \hat{q}_1}{\partial q_0}^T \left(p_{1,H_d^+} - D_2L_d(q_0, \hat{q}_1) \right) + \hat{p}_0 \\
&= \hat{p}_0,
\end{align*}
where the last line follows by the assumption $p_{1,H_d^+}=D_2L_d(q_0,\hat{q}_1)$. Therefore, $p_0=\hat{p}_0$, which then implies $-p_0 = D_1L_d(q_0,\hat{q}_1)$ and consequently $\hat{q}_1 = q_{1,L_d}$. Applying the next discrete Legendre transform yields,
\begin{align*}
q_{1,H_d^+} &= D_2H_d^+(q_0, p_{1,H_d^+}) \\
&= \frac{\partial \hat{q}_1}{\partial p_{1,H_d^+} }^T p_{1,H_d^+} + \hat{q}_1 - \left( \frac{\partial \hat{q}_1}{\partial p_{1,H_d^+} }^T D_2L_d(q_0, \hat{q}_1) \right) \\
&= \frac{\partial \hat{q}_1}{\partial p_{1,H_d^+} }^T \left( p_{1,H_d^+} - D_2L_d(q_0, \hat{q}_1) \right) + \hat{q}_1 \\
&= \hat{q}_1.
\end{align*}
Therefore, $\hat{q}_1 = q_{1,H_d^+}$, which implies $q_{1,L_d} = q_{1,H_d^+}$. Now we have,
\begin{align*}
p_{1,H_d^+} &= D_2L_d(q_0,\hat{q}_1) \\
&= D_2L_d(q_0, q_{1,L_d}) \\
&= p_{1,L_d}.
\end{align*}
\end{proof}


\bibliography{vi_review}
\bibliographystyle{plainnat}

\end{document}